\documentclass[twoside]{article}
\usepackage[margin=1in]{geometry}
\usepackage[english]{babel}
\usepackage{amsmath}
\usepackage{amsfonts}
\usepackage{amssymb}
\usepackage{amsthm}
\usepackage[nocompress]{cite}
\usepackage{graphicx}
\usepackage{xcolor}
\usepackage{epstopdf}
\usepackage{algorithm}
\usepackage{algorithmic}
\usepackage{footnote}
\usepackage{multirow}
\usepackage{hyperref}
\usepackage[capitalise]{cleveref}
\usepackage{bm}
\usepackage{bbm}

\crefformat{equation}{(#2#1#3)}
\crefrangeformat{equation}{(#3#1#4) to~(#5#2#6)}
\crefmultiformat{equation}{(#2#1#3)}{ and~(#2#1#3)}{, (#2#1#3)}{ and~(#2#1#3)}
\newtheorem{theorem}{Theorem}[section]
\newtheorem{remark}[theorem]{Remark}
\newtheorem{proposition}[theorem]{Proposition}
\newtheorem{corollary}[theorem]{Corollary}
\crefname{remark}{Remark}{Remarks}
\crefname{proposition}{Proposition}{Propositions}
\crefname{corollary}{Corollary}{Corollaries}

\usepackage{fancyhdr}
\title{Surrogate modeling of resonant behavior in scattering problems through adaptive rational approximation and sketching\thanks{This research was funded in part by the Austrian Science Fund (FWF) projects 10.55776/F65 and 10.55776/P33477 (IP).}}
\newcommand{\email}[1]{\protect\href{mailto:#1}{#1}}
\author{Davide Pradovera\thanks{Department of Mathematics, Stockholm University, Stockholm, Sweden (\email{davide.pradovera@math.su.se}).}
\and Ralf Hiptmair\thanks{Seminar f\"ur Angewandte Mathematik, Department of Mathematics, ETH Zurich, Zurich, Switzerland (\email{hiptmair@sam.math.ethz.ch}).}
\and Ilaria Perugia\thanks{Faculty of Mathematics, University of Vienna, Vienna, Austria
  (\email{ilaria.perugia@univie.ac.at}).}}

\newcommand{\bN}{\mathbb N}
\newcommand{\bR}{\mathbb R}
\newcommand{\bC}{\mathbb C}
\newcommand{\vv}{\mathbf v}
\newcommand{\vx}{\mathbf x}
\newcommand{\vy}{\mathbf y}
\DeclareMathOperator*{\real}{Re}
\DeclareMathOperator*{\imag}{Im}
\newcommand{\abs}[1]{\left|#1\right|}
\newcommand{\norm}[1]{\left\|#1\right\|}
\newcommand{\normx}{\norm{\vx}}
\newcommand{\normsp}[1]{\norm{#1}_{\textup{sp}}}
\newcommand{\utot}{u_{\textup{tot}}}
\newcommand{\uinc}{u_{\textup{inc}}}
\newcommand{\uin}{u_{\textup{i}}}
\newcommand{\uh}{\mathbf{u}}
\newcommand{\gh}{\mathbf{g}}
\newcommand{\uinhD}{(u_{\textup{i}})_h}
\newcommand{\uinhN}{(\partial_\nu u_{\textup{i}})_h}
\newcommand{\uout}{u_{\textup{s}}}
\newcommand{\nin}{n_{\textup{i}}}
\newcommand{\Ltwofunc}[1]{L^2(#1)}
\newcommand{\Ltwobd}{\Ltwofunc{\Gamma}}
\newcommand{\Ltwobdtwo}{[\Ltwobd]^2}
\newcommand{\Honefunc}[1]{H^1(#1)}
\newcommand{\Honelocall}[1]{\Honelocfunc{\bR^d}}
\newcommand{\Honein}{\Honefunc{\Omega}}
\newcommand{\Honelocfunc}[1]{H_{\textup{loc}}^1(#1)}
\newcommand{\Honelocout}{\Honelocfunc{\bR^d\setminus\overline{\Omega}}}
\newcommand{\Hhalffunc}[1]{H^{1/2}(#1)}
\newcommand{\Hhalfbd}{\Hhalffunc{\Gamma}}
\newcommand{\Hmhalffunc}[1]{H^{-1/2}(#1)}
\newcommand{\Hmhalfbd}{\Hmhalffunc{\Gamma}}

\newcommand{\bigO}[1]{\mathcal{O}\left(#1\right)}

\newcommand{\calSingle}[1]{\mathsf{V}(#1)} 
\newcommand{\calDouble}[1]{\mathsf{K}(#1)} 
\newcommand{\OpSingle}[1]{\mathcal{V}(#1)}
\newcommand{\OpDouble}[1]{\mathcal{K}(#1)}
\newcommand{\OpHyper}[1]{\mathcal{W}(#1)}
\newcommand{\calA}[1]{\mathcal{A}(#1)}
\newcommand{\calB}[1]{\mathcal{B}(#1)}

\newcommand{\calH}[1]{\mathcal{H}(#1)}
\newcommand{\calL}[1]{\mathcal{L}(#1)}
\newcommand{\calP}{\mathcal{P}}
\newcommand{\calX}{\mathcal{X}}
\newcommand{\calXh}{\calX_h}

\newcommand{\Nh}{N_h}
\newcommand{\NK}{N_K}
\newcommand{\nK}{n_K}
\newcommand{\hatnK}{\widehat n_K}

\newcommand{\uA}[1]{\textup{A}(#1)}
\newcommand{\uAk}{\uA{k}}
\newcommand{\uB}[1]{\textup{B}(#1)}
\newcommand{\uBk}{\uB{k}}
\newcommand{\uC}[1]{\textup{C}(#1)}
\newcommand{\uCk}{\uC{k}}
\newcommand{\uCz}{\textup{C}}
\newcommand{\uH}[1]{\textup{H}(#1)}
\newcommand{\uHk}{\uH{k}}
\newcommand{\uHz}{\textup{H}}
\newcommand{\uM}{\textup{M}}
\newcommand{\uL}[1]{\textup{L}(#1)}
\newcommand{\uLk}{\uL{k}}
\newcommand{\uLz}{\textup{L}}
\newcommand{\uRz}{\textup{R}}
\newcommand{\vb}{\mathbf{b}}
\newcommand{\vbh}{\vb}
\newcommand{\uBh}{\textup{B}}
\newcommand{\vell}{\bm\xi}

\newcommand{\LINECOMMENT}[1]{\STATE{\# \textit{#1}}}

\begin{document}
\maketitle

\begin{abstract}
    This paper describes novel algorithms for the identification of (almost-)resonant behavior in scattering problems. Our methods, relying on rational approximation, aim at building surrogate models of what we call ``field amplification'', defined as the norm of the solution operator of the scattering problem, which we express through boundary-integral equations. To provide our techniques with theoretical foundations, we first derive results linking the field amplification to the spectral properties of the operator that defines the scattering problem. Such results are then used to justify the use of rational approximation in the surrogate-modeling task. Some of our proposed methods apply rational approximation in a ``standard'' way, building a rational approximant for either the solution operator directly or, in the interest of computational efficiency, for a randomly ``sketched'' version of it. Our other ``hybrid'' approaches are more innovative, combining rational-approximation-assisted root-finding with approximation using radial basis functions. Three key features of our methods are that (i) they are agnostic of the strategy used to discretize the scattering problem, (ii) they do not require any computations involving non-real wavenumbers, and (iii) they can adjust to different settings through the use of adaptive sampling strategies. We carry out some numerical experiments involving 2D scatterers to compare our approaches. In our tests, two of our approaches (one standard, one hybrid) emerge as the best performers, with one or the other being preferable, depending on whether emphasis is placed on accuracy or efficiency.
\end{abstract}



\section{Introduction}\label{sec:intro}
This work is concerned with the study of resonating behavior arising in wave-scattering problems, which, in frequency domain, are modeled through the Helmholtz equation. It is well known from Fredholm theory that, as long as \emph{Sommerfeld radiation conditions} are enforced, under mild assumptions on the media scattering problems admit a unique solution at any nonzero \emph{real} frequency. However, in practical terms the scattering problem may no longer be well-posed as a consequence of what we informally call ``almost resonances'', namely, frequencies where the inf-sup constant of the problem is dangerously close to 0.

From a spectral viewpoint, such almost resonances can be interpreted as the ``\emph{real-axis} shadows'' of the actual resonances that scattering problems display in the \emph{complex} plane. As we discuss below in more detail, such complex resonances with nonzero imaginary parts are the solutions of certain nonlinear eigenvalue problems. Phenomenologically speaking, they correspond to a resonating behavior caused by the addition of a negative dissipation coefficient, i.e., a positive \emph{reaction} term, to the wave equation. This ultimately leads to a blow-up in the scattered field.

Almost-resonance problems are usually more severe at large frequencies, since, in the large-frequency regime, the complex resonances of the scattering problem get arbitrarily close to the real axis. A good understanding of these phenomena is paramount in many fields. For instance, resonance is sometimes exploited as a way to increase the power of a signal. On the other hand, one may want to avoid resonance to  make a physical system more stable or well-behaved. This motivates our interest in this problem.

Incidentally, we also mention that almost resonances have important consequences on the \emph{numerical conditioning} of the scattering problem after discretization. Thus, spotting almost resonances may be useful for understanding the behavior of numerical methods, e.g., in terms of convergence of iterative linear-system solvers.

At almost-resonating frequencies, the scattered wave may be rather large (in some norm) even for small incident waves. This issue may be equivalently stated in terms of a large operator norm (in suitable spaces) of the \emph{solution operator} of the scattering problem, which, at fixed frequency, maps an incident wave to a scattered wave. In this work, we use this norm, under the name of ``\emph{field amplification}'', as a quantitative measure of almost-resonating behavior.

For a given scattering problem, our ultimate task is to evaluate the field amplification over a prescribed range of frequency values. A key observation is that the cost of evaluating field amplification over a whole frequency range may be rather high, especially if we use a high resolution both in space and frequency. Indeed, one has to assemble the discrete solution operator and then determine its norm (e.g., by SVD) at each frequency belonging to a sampling grid that covers the target frequency range.

To lessen this burden, we rely on \emph{surrogate modeling}, which suggests to replace the expensive high-fidelity scattering model with a more economical low(er)-fidelity surrogate. The construction of such surrogate (the so-called ``training'' phase) is generally not cheap, since it must involve several evaluations of the high-fidelity model. However, it must be carried out only once, independently of the desired number of frequency-domain samples. Ignoring further training costs for simplicity, the complexity of a surrogate-based method for approximating scattering problems is often $\bigO{c_{\textup{high-fidelity}}n+c_{\textup{surrogate}}N}$, with $n$ denoting the numbers of training samples and $N$ denoting the number of target frequency evaluations. Compared to the ``brute force'' approach where the high-fidelity model is evaluated sequentially, computational time can be saved as long as the sampling is parsimonious ($n\ll N$) and the surrogate is efficient ($c_{\textup{surrogate}}\ll c_{\textup{high-fidelity}}$).

In recent years, surrogate modeling has often been applied with the objective of reducing the simulation costs linked to scattering problems \cite{li21,vi11,zimmerling23,bruno24,pradovera22,hanke12,grubisic23,chen10}. In order to build an effective surrogate, a careful balance usually has to be struck between \emph{a priori} knowledge of the high-fidelity problem's structure and \emph{a posteriori} information obtained from data. Both \emph{a priori} and \emph{a posteriori} information must be leveraged in the quest for an efficient method, i.e., keeping $n$ and $c_{\textup{surrogate}}$ low while achieving an acceptable approximation accuracy. The \emph{a priori} knowledge on the problem is crucial for this, since it is what enables the selection of the ``approximation class'' within which the surrogate is then sought through a data-driven strategy. In the context of surrogate modeling, the approximation class is crucial for the effectiveness of a numerical algorithm.

As an example of this, we mention \cite{grubisic23}, whose specific target is the approximation of field amplification just like in this work. The authors of \cite{grubisic23} achieve efficiency with their surrogate-based method by a careful combination of adaptive sampling (keeping $n$ low) and piecewise-polynomial approximation of the reciprocal of the field amplification (keeping $c_{\textup{surrogate}}$ low).

However, the effectiveness of the algorithm in \cite{grubisic23} is ultimately limited by the fact that the chosen approximation class only partially exploits the structure of the high-fidelity problem. Indeed, the method relies only on the piecewise-smooth nature of the field amplification but not on the meromorphy of the underlying scattering problem, cf.~\cref{sec:spectral}. In this work, we try to remedy this, by developing methods that better profit from both features. To this aim, we appeal to some key results from operator theory to draw a key connection between our approximation target, namely, field amplification, and a certain nonlinear eigenvalue problem induced by the scattering problem's differential operator.

A key aspect of our work is that, in building a surrogate for the field amplification, we assume that we can sample at \emph{only real} wavenumbers. This sets our work apart from, e.g., \cite{beyn12,bruno24}. Fundamentally, we do this because field amplification is a real quantity, related to real wavenumbers, so it should be possible to approximate it by sampling real wavenumbers only. Moreover, and from a more practical viewpoint, non-real wavenumbers are sometimes out of reach due to limitations in the numerical software that is used to discretize the scattering problem in the first place, and ultimately to compute field amplification. Indeed, adding an imaginary component to the wavenumber is arguably nonphysical, so that, by design, it may not be allowed by, e.g., finite- or boundary-elements libraries.

We stress that the focus of this paper is on the derivation and empiric validation of numerical algorithms. We conduct some mathematical analysis, but have to point out that some aspects of our new algorithm are still beyond the scope of rigorous theory.

The outline of the paper is as follows. After introducing the target scattering problem and its boundary-element discretization in \cref{sec:bem}, its above-mentioned \emph{spectral} properties are analyzed in \cref{sec:spectral}. Our novel methods are then outlined in \cref{sec:approx,sec:hybrid}, with the former discussing more ``standard'' approximation strategies, and the latter presenting a more ``innovative'' hybrid technique. Numerical results testing and comparing our proposed methods follow in \cref{sec:numexp}, demonstrating the effectiveness of the hybrid technique introduced in \cref{sec:hybrid}.

\section{Boundary-integral equations}\label{sec:bem}

We consider a general scattering problem in $d$ dimensions, with $d\in\{2,3\}$, modeled through the Helmholtz equation with outgoing Sommerfeld radiation conditions
\begin{equation}\label{eq:strongform}
    \begin{cases}
        -\Delta\utot-(kn)^2\utot=0\quad&\textup{in }\bR^d,\\
        \left.\nabla(\utot-\uinc)\right\rvert_\vx\cdot\frac{\vx}{\normx}-\textup{i}kn_\infty\left.(\utot-\uinc)\right\rvert_\vx=o\left(\normx^{\frac{1-d}{2}}\right)\quad&\textup{as }\normx\to\infty.
    \end{cases}
\end{equation}
The unknown is the \emph{total wave} $\utot:\bR^d\to\bC$, $\textup{i}=\sqrt{-1}$ is the imaginary unit, and $\uinc:\bR^d\to\bC$ is a given \emph{incident} wave. Moreover, $n:\bR^d\to\bR_{>0}$ is the index of refraction, which is assumed to attain positive constant values $\nin$ and~$n_\infty$ inside and outside, respectively, of a suitable bounded, Lipschitz domain $\Omega\subset\bR^d$. Without loss of generality, by a scaling argument, we can set $n_\infty:=1$ outside $\Omega$, so that
\begin{equation}\label{eq:refraction}
    n(\vx)=\begin{cases}
        \nin\in\bR_{>0}\quad&\textup{for }\vx\in\Omega,\\
        1\quad&\textup{for }\vx\notin\Omega.
    \end{cases}
\end{equation}
Hence, the boundary $\Gamma:=\partial\Omega$ represents an interface between two materials with different, but uniform, acoustic properties. Generalizations to non-constant $\nin$ are possible.

In the following, we will sometimes explicitly denote the dependence of waves $u_\bullet:\bR^d\to\bC$ on the wavenumber $k\in\bR$, through the notations $u_\bullet(k)$ or $u_\bullet(\vx;k)$.

Under the assumption that the incident wave satisfies the Helmholtz equation with index of refraction $1$, i.e.,
\begin{equation}\label{eq:incident}
    -\Delta\uinc-k^2\uinc=0\quad\textup{in }\bR^d,
\end{equation}
we can equivalently cast the problem in \emph{transmission} form, denoting by $\uin$ the total field within $\Omega$ and by $\uout$ the scattered field outside $\Omega$:
\begin{equation}\label{eq:transmissionstrongform}
    \begin{cases}
        -\Delta\uin-(k\nin)^2\uin=0\quad&\textup{in }\Omega,\\
        -\Delta\uout-k^2\uout=0\quad&\textup{in }\bR^d\setminus\overline\Omega,\\
        \uin=\uout+g_D\quad&\textup{on }\Gamma,\\
        \partial_\nu\uin=\partial_\nu\uout+g_N\quad&\textup{on }\Gamma,\\
        \left.\nabla\uout\right\rvert_\vx\cdot\frac{\vx}{\normx}-\textup{i}k\uout(\vx)=o\left(\normx^{\frac{1-d}{2}}\right)\quad&\textup{as }\normx\to\infty.
    \end{cases}
\end{equation}
Above, we have introduced $\nu$, the outward-pointing normal to $\Gamma$, as well as
\begin{equation*}
    \utot(\vx;k)=:\begin{cases}
        \uin(\vx;k)\quad&\textup{for }\vx\in\Omega,\\
        \uout(\vx;k)+\uinc(\vx;k)\quad&\textup{for }\vx\notin\Omega.
  \end{cases}
\end{equation*}
The jump data $(g_D,g_N)\in\Hhalfbd\times\Hmhalfbd$ are defined based on the incident wave, through Dirichlet and Neumann trace operators, namely,
\begin{equation*}
    g_D=\left.\uinc\right\rvert_{\Gamma}\quad\textup{and}\quad g_N=\left.(\partial_\nu\uinc)\right\rvert_{\Gamma}.
\end{equation*}

It is known that, under the above assumptions, with the addition of smoothness requirements on $\Gamma$ and $\uinc$, the scattering problems in \cref{eq:strongform,eq:transmissionstrongform} admit weak solutions $\utot\in\Honelocfunc{\bR^d}$ and $(\uin,\uout)\in\Honein\times\Honelocout$, respectively, for all nonzero wavenumbers $k\in\bR\setminus\{0\}$ \cite{grubisic23}.

\subsection{Second-kind boundary-integral equations}

Closely following \cite[Section 2]{grubisic23} and \cite[Section~1.3]{HMS21}, we rewrite the transmission problem \cref{eq:transmissionstrongform} as a boundary-integral equation. We leverage representations of $\uin$ and $\uout$ in terms of \emph{layer potentials}, via the free-space Green's function for the Helmholtz operator $\Delta + \kappa^{2}$, namely,
\begin{equation*}
    G(\vx;\kappa) :=\begin{cases}
        \frac{\textup{i}}4H_0(\kappa\normx)\quad&\textup{if }d=2,\\
        \frac1{4\pi\normx}\exp(\textup{i}\kappa\normx)\quad&\textup{if }d=3,\\
    \end{cases}
\end{equation*}
with $H_0$ the index-0 Hankel function of first kind. Note, in particular, that the Green's function depends on $\kappa$ in a very \emph{nonlinear} but \emph{holomorphic} way. The (unique) solution of \cref{eq:transmissionstrongform} admits a layer potential representation \cite[Eq.~(5)]{grubisic23}
\begin{equation}
  \label{eq:green}
    \begin{cases}
      \uin=-\calSingle{k\nin}(\partial_\nu\uin|_\Gamma)-
      \calDouble{k\nin}(\uin|_\Gamma)\quad&\text{in }\Omega,\\
      \uout=\calSingle{k}(\partial_\nu\uout|_\Gamma)+
      \calDouble{k}(\uout|_\Gamma)&\text{in }\bR\setminus\overline{\Omega},
    \end{cases}
\end{equation}
with the ``single-layer'' and ``double-layer'' linear integral operators defined in a parameter-dependent way as
\begin{subequations}
  \label{eq:opssingledouble}
  \begin{equation}
    \label{eq:spo}
      \calSingle{\kappa}:\Hmhalfbd\ni\psi\mapsto\int_{\Gamma}G(\cdot-\vy;\kappa)\psi(\vy)\textup{d}S(\vy) \in\Honelocall\;,
  \end{equation}
  \begin{equation}
      \label{eq:dpo}
      \calDouble{\kappa}:\Hhalfbd\ni\psi\mapsto-\int_{\Gamma}\partial_{\nu(\vy)}G(\cdot-\vy;\kappa)\phi(\vy)\textup{d}S(\vy)\in\Honelocall\;.
  \end{equation}
\end{subequations}
Applying Dirichlet and Neumann traces to the single- and double-layer potential from \cref{eq:opssingledouble} yields four \emph{boundary-integral operators}, whose integral representations read
  \begin{gather}
    \label{eq:biop}
    \begin{aligned}
      \OpSingle{\kappa}\psi & := \int_{\Gamma}G(\cdot-\vy;\kappa)\psi(\vy)\textup{d}S(\vy),\\
      \OpDouble{\kappa}\phi & :=
      -\int_{\Gamma}\partial_{\nu(\vy)}G(\cdot-\vy;\kappa)\phi(\vy)\textup{d}S(\vy),\\
      \OpDouble{\kappa}^{*}\phi & :=
      -\int_{\Gamma}\partial_{\nu(\vx)}G(\cdot-\vy;\kappa)\phi(\vy)\textup{d}S(\vy),\\
      \OpHyper{\kappa}\phi & :=
      -\int_{\Gamma}\partial_{\nu(\cdot)}\partial_{\nu(\vy)}G(\cdot-\vy;\kappa)\phi(\vy)\textup{d}S(\vy),
    \end{aligned}
  \end{gather}
see, e.g., \cite[Section~3.3]{sauter11}. These are the building blocks for the second-kind single-trace boundary-integral equation satisfied by the Dirichlet and Neumann traces of the solution of the scattering transmission problem \cref{eq:transmissionstrongform}. More precisely, applying both trace operators to the representation formula \cref{eq:green}, appealing to jump relations \cite[Section~3.1.1]{sauter11}, and taking into account transmission conditions, we obtain
\begin{equation}\label{eq:bdintform}
    \calA{k}
    \begin{bmatrix}
        \uin|_\Gamma\\ \partial_\nu\uin|_\Gamma
    \end{bmatrix}=\calB{k}\begin{bmatrix}
      g_D\\ g_N
    \end{bmatrix},
\end{equation}
with the compound boundary-integral operators ($\mathcal{I}$ being the identity operator in $\Ltwobd$)
\begin{equation*}
    \calA{k}:=\begin{pmatrix}
        \mathcal{I}+\OpDouble{k\nin}-\OpDouble{k} & \OpSingle{k}-\OpSingle{k\nin}\\
        \OpHyper{k}-\OpHyper{k\nin} & \mathcal{I}+\OpDouble{k\nin}^*-\OpDouble{k\nin}^*
    \end{pmatrix}
\end{equation*}
and
\begin{equation*}
    \calB{k}:=\begin{pmatrix}
        \frac12\mathcal{I}-\OpDouble{k} & \OpSingle{k}\\
        \OpHyper{k} & \frac12\mathcal{I}+\OpDouble{k}^*
    \end{pmatrix}.
\end{equation*}

The former operator $\calA{k}$ is bounded on $\Ltwobdtwo$ \cite[Lemma~1]{grubisic23}. Other properties of this formulation of the scattering problem, including its well-posedness, are discussed in \cref{sec:spectral}. Another key feature of \cref{eq:bdintform} is the nonlinear but holomorphic dependence of the two operators $\calA(k)$ and $\calB(k)$ on the wavenumber $k$.

\begin{remark}\label{rem:RHSplanewave}
    In several cases of practical interest, $(g_D,g_N)$ contains Dirichlet and Neumann traces on $\Gamma$ of a field $\uinc$ that locally solves the Helmholtz equation with unit refraction index, cf.~\cref{eq:incident}. One can verify that such boundary data is an eigenfunction of $\calB{k}$ with unit eigenvalue \cite[Theorem 2.6]{claeys13}, i.e.,
    \begin{equation*}
        \calB{k}\begin{bmatrix}
            g_D\\ g_N
        \end{bmatrix}=\begin{bmatrix}
            g_D\\ g_N
        \end{bmatrix}.
    \end{equation*}
    Thus, one may sometimes simplify \cref{eq:bdintform} by replacing $\calB{k}$ with the identity operator.
\end{remark}

\subsection{Galerkin boundary-element discretization}\label{sec:bemmat}

We consider the boundary-integral equation \cref{eq:bdintform} in $\Ltwobdtwo$ and use the $\Ltwobd$ inner product to cast it into a variational form, which is the starting point for Galerkin discretization. Next, we pick a finite-dimensional subspace $\calXh$ of $\Ltwobd$, with $\Nh=\dim\calXh$, and use $\calXh^2$ as trial and test spaces. Then we choose a basis $\{\psi_i\}_{i=1}^{\Nh}$ of $\calXh$, which induces a basis $\{[\psi_{i},0]^{\top}\}_{i=1}^{\Nh}\cup\{[0,\psi_{i}]^{\top}\}_{i=1}^{\Nh}$ of $\calXh^2$. We use this latter basis to convert \cref{eq:bdintform} into a $(2\Nh)\times(2\Nh)$ linear system of equations
\begin{equation}\label{eq:bem}
    \uAk\uh=\uBk\gh,
\end{equation}
where the vector $\gh$ is defined as the coefficients (with respect to the chosen basis) of the $\Ltwobdtwo$-orthogonal projection of $[g_D,g_N]^\top$ onto $\calXh^2$. Moreover, \cref{eq:bem} features the matrix $\uBk$, defined entry-wise as
\begin{equation*}
    \uBk_{ij}:=\begin{cases}
        \Big(\big({\textstyle\frac12}\mathcal{I}-\OpDouble{k}\big)\psi_j,\psi_i\Big)\quad&\text{if }i\leq \Nh\text{ and }j\leq \Nh,\\
        \Big(\OpSingle{k}\psi_{j-\Nh},\psi_i\Big)\quad&\text{if }i\leq \Nh\text{ and }j>\Nh,\\
        \Big(\OpHyper{k}\psi_j,\psi_{i-\Nh}\Big)\quad&\text{if }i>\Nh\text{ and }j\leq \Nh,\\
        \Big(\big({\textstyle\frac12}\mathcal{I}+\OpDouble{k}^*\big)\psi_{j-\Nh},\psi_{i-\Nh}\Big)\quad&\text{if }i>\Nh\text{ and }j>\Nh,\\
    \end{cases}
\end{equation*}
with $(u,v)$ denoting the $\Ltwobd$-inner product, as well as $\uAk=\uM+\uBk-\uB{k\nin}$, with $\uM$ being the $\Ltwobdtwo$-mass matrix, defined entry-wise as
\begin{equation}
    \label{eq:bemmass}
    \uM_{ij}:=\begin{cases}
        \big(\psi_j,\psi_i\big)\quad&\text{if }i\leq \Nh\text{ and }j\leq \Nh,\\
        \big(\psi_{j-\Nh},\psi_{i-\Nh}\big)\quad&\text{if }i>\Nh\text{ and }j>\Nh,\\
        0\quad&\text{otherwise}.\\
    \end{cases}
\end{equation}
Note that we are glossing over some ``variational crimes'' committed through the inevitable use of numerical quadrature. We refer to \cite[Chapter~5]{sauter11} for details on this issue and a comprehensive analysis.

Once the vector $\uh=\uh(k)\in\bC^{2\Nh}$ solving \cref{eq:bem} has been found, approximations of the Dirichlet and Neumann traces can be recovered through linear combinations of the basis functions:
\begin{gather*}
    u_{i}|_\Gamma \approx\uinhD=\sum_{j=1}^{\Nh}(\uh)_j\psi_j\quad\textup{and}\quad
    \partial_{\nu}u_{i}|_\Gamma \approx \uinhN=\sum_{j=1}^{\Nh}(\uh)_{\Nh+j}\psi_j.
\end{gather*}
The \emph{boundary-element method} (BEM) relies on a particular construction of the discrete space $\calXh$ built by first triangulating $\Gamma$ and then using piecewise polynomial spaces on the resulting triangulation. Details are given in \cite[Chapter~4]{sauter11}.

\begin{remark}\label{rem:bemwithoutB}
    In \cref{rem:RHSplanewave}, we have mentioned how one can sometimes simplify the boun\-dary-integral formulation by replacing the operator $\calB{k}$ with the identity. In the corresponding discrete formulation, one must replace $\uBk$ with the mass matrix $\uM$ defined in \cref{eq:bemmass}.
\end{remark}

\section{Scattering spectrum and field amplification}\label{sec:spectral}

The well-posedness of the boundary-integral formulation \cref{eq:bdintform} and of its BE discretization \cref{eq:bem} may be determined by a Fredholm-alternative argument \cite[Section~2.1.4]{sauter11}, based on the observation that, for all $k\in\bR$, $\calA{k}$ is a compact perturbation of the identity operator in $\Ltwobdtwo$. This makes it an operator of index zero, see, e.g., \cite{sloan92,sauter11,grubisic23}. Thanks to the injectivity of $\calA{k}$ for any nonzero $k$ \cite[Lemma~1.7]{HMS21}, one can conclude that \cref{eq:bdintform} is well-posed for all nonzero wavenumbers $k$. From \cite[Section~4.2.3]{sauter11} we learn that, provided that $\calXh$ possesses sufficiently good $\Ltwobd$-approximation properties, the discretized boundary-integral equation \cref{eq:bem} is also well-posed.

The homogeneous version of the continuous and discretized boundary-integral equations
\begin{gather}
  \label{eq:bdintformevp}
    \calA{k}
    \begin{bmatrix}
        \uin\\ \partial_\nu\uin
    \end{bmatrix}=0\quad\textup{and}\quad\uAk\uh=0,
\end{gather}
obtained by setting $g_D=g_N=0$, has some relevance on its own. Indeed, it can be interpreted as a \emph{nonlinear eigenvalue problem}, with $k$ and $[\uin,\partial_\nu\uin]^\top$ playing the parts of eigenvalue and eigenfunction, respectively. This problem has been extensively studied, and many of its theoretical properties have been explored before. For instance, eigenvalues must have negative imaginary parts \cite[Appendix A]{moiola19}, with the exception of $0$, which admits the constant eigenfunction $[\uin,\partial_\nu\uin]^\top\equiv[1,0]^\top$. Moreover, as we show in the next section, the set of eigenvalues is discrete and has no finite accumulation points.

\subsection{Some properties of the solution operator}

Most results presented in this section apply regardless of whether the infinite-dimensional boundary-integral formulation \cref{eq:bdintform} or its BE discretization \cref{eq:bem} is considered. As such, we employ the symbol $\calL{k}$ to denote either $\calA{k}$ or $\uAk$.

We first recall Keldysh's theorem for the inverse of a parameter-dependent operator, generalizing \cite[Theorem 2.8]{guttel17} to the operator setting. See also \cite[Chapter 5]{gohberg69} and references therein for more details.

\begin{theorem}\label{th:keldysh}
    Let $K\subset\bC$ be compact. Let $\calL{k}$ be an operator-valued function of $k$, holomorphic (with respect to $k$) over $K$. We assume that, for all $k\in K$, $\calL{k}$ is a Fredholm operator of index zero over some Hilbert space $\calX$, e.g., $(\calL{k}-I):\calX\to\calX$ is compact, with $I$ the identity operator over $\calX$. If $\calL{k}^{-1}$ exists for at least one value of $k\in K$, then
    \begin{equation}\label{eq:keldyshgen}
        \calL{k}^{-1}=\sum_{i=1}^M\sum_{j=1}^{m_i}\frac{\calP_{i,j}}{(k-\lambda_i)^j}+\calH{k}\quad\textup{for all }k\in K\setminus\{\lambda_i\}_{i=1}^M.
    \end{equation}
    Above, we have introduced several new objects, which we characterize as follows:
    \begin{itemize}
        \item $\Lambda_K=\{\lambda_i\}_{i=1}^M\subset K$, whose elements are assumed pairwise distinct without loss of generality, is the \emph{finite} set of values of $k\in K$ where $\calL{k}$ is not invertible.
        \item For each $i$, $m_i\geq1$ is the length of the longest \emph{chain of generalized eigenvectors} corresponding to $\lambda_i$. The positive integer $m_i$ represents the \emph{order} of $\lambda_i$ as a \emph{pole} of $\calL{k}^{-1}$, in a complex-analysis sense.
        \item For each $i$ and $j$, $\calP_{i,j}$ is an (oblique) projection operator whose range lies within the generalized eigenspace corresponding to the eigenvalue $\lambda_i$; for each $i$, $\calP_{i,m_i}$ is nonzero.
        \item $\calH{k}$ is an operator-valued function, holomorphic (with respect to $k$) over $K$.
    \end{itemize}
    Note, in particular, that $\Lambda_K$ and $\calH{k}$ are $K$-dependent.
\end{theorem}

\begin{remark}\label{rem:jordanchain}
    Keldysh's theorem is founded on complex analysis and operator theory. Specifically, the ``chains of generalized eigenvectors'', which allow the definition of the pole orders $m_i$, are the nonlinear extension of ``Jordan chains'', which play a pivotal role in linear eigenvalue problems for non-normal (non-diagonalizable) matrices. In the framework of \cref{th:keldysh}, we call $\{v_j\}_{j=0}^{m-1}\subset\calX\setminus\{0\}$ a length-$m$ chain of generalized eigenvectors associated to $k=\lambda$ for $\calL{k}$ if
    \begin{equation*}
        \sum_{n=0}^j\left.\left(\partial_k^{j-n}\calL{k}\right)\right|_{k=\lambda}v_n=0\quad\textup{for all }j=0,\ldots,m-1.
    \end{equation*}
\end{remark}

Since the boundary-integral operators \cref{eq:biop} and their BE discretizations appearing in \cref{eq:bem} depend on $k$ in a holomorphic way, the above theorem applies to both infinite- and finite-dimensional cases, by defining $\calX=\Ltwobdtwo$ and $\calX=\bC^{2\Nh}$, respectively.

As already mentioned, $\calL{k}$ is invertible for all nonzero wavenumbers $k$ with non-negative imaginary part. As a consequence, $\calL{k}^{-1}$ exists for at least one value of $k\in\bC$, and \cref{th:keldysh} states that $\calL{k}^{-1}$ is \emph{meromorphic} over $\bC$.
By applying \cref{th:keldysh} over compact subsets of the upper half of the complex plane, one can conclude that $\calL{k}^{-1}$ is well-defined and holomorphic there. In addition, the set $\Lambda\subset\bC$ of wavenumbers $k$ such that $\calL{k}$ is not invertible cannot have any finite accumulation points.

Based on these two properties, a simple analytic-continuation argument shows that, given a real interval not containing $0$, $[k_{\textup{min}},k_{\textup{max}}]\subset\bR\setminus\{0\}$, there exists a finite-thickness ``strip''
\begin{equation}\label{eq:strip}
    \{z\in\bC,k_{\textup{min}}\leq\real(z)\leq k_{\textup{max}},-\varepsilon\leq\imag(z)\leq\varepsilon\}
\end{equation}
where $\calL{k}^{-1}$ is well-defined and holomorphic, as long as $\varepsilon>0$ is sufficiently small. Specifically, $(-\varepsilon)$ must be larger than the imaginary part of any eigenvalues $\lambda$ whose real part is within the chosen wavenumber range. However, in most settings, one can show that eigenvalues get arbitrarily close to the real line \cite{moiola19}, so that the above-mentioned ``strip of holomorphy'' must necessarily get thinner as one moves to larger (in magnitude) real wavenumber ranges.

This last feature can be equivalently stated in terms of possibility of approximation. Given a real interval, the thickness of the strip of holomorphy is directly related to the ``ease'' with which $\calL{k}^{-1}$ can be approximated by polynomials. Roughly speaking, the rate of convergence of the error achieved, e.g., by approximating $\calL{k}^{-1}$ with Chebyshev polynomials, improves as the strip of holomorphy becomes thicker. For large(r) wavenumbers, the thinness of the strip of holomorphy behooves one to discard polynomials altogether, resorting to other approximation strategies instead. In particular, given the meromorphy of $\calL{k}^{-1}$, approximation by \emph{rational functions} appears natural, cf.~\cref{sec:sub:rat}.

\subsection{Field amplification}\label{sec:sub:field}

We are finally ready to properly introduce the quantity whose approximation is the target of our work, to which we refer as ``field amplification'' $\phi=\phi(k)$. We informally define it as the (operator) norm of the solution operator of the scattering problem, mapping incident data to the solution field. More rigorously, employing the boundary-integral form with the simplification described in \cref{rem:RHSplanewave}, we define
\begin{equation}\label{eq:amplificationinf}
    \phi(k)=\norm{\calA{k}^{-1}}_{\Ltwobdtwo\to\Ltwobdtwo}=\sup_{(g_D,g_N)\in\Ltwobdtwo\setminus\{(0,0)\}}\norm{\calA{k}^{-1}\begin{bmatrix}
        g_D\\g_N
    \end{bmatrix}}_{\Ltwobdtwo}\bigg/\norm{\begin{bmatrix}
        g_D\\g_N
    \end{bmatrix}}_{\Ltwobdtwo}.
\end{equation}

As the name suggests, field amplification measures how large the norm of the scattered solution field can be, with respect to that of the incident field. This quantity is of great importance in several settings, e.g., to study the desired or undesired resonant behavior of physical systems. Moreover, field amplification has an impact on the numerics of solving scattering problems, since it is related to the conditioning of the BE problem. If $\phi(k)$ is large, significant errors may be incurred due to round-off or other inaccuracies in the data, when numerically computing the scattered field. For instance, knowledge of $\phi$ (even if approximate) allows one to draw quantitative conclusions on the magnitude of perturbations in the scattered field, resulting from perturbations in the incoming wave. This is of particular interest, e.g., in uncertainty quantification.

Computing $\phi$ in practical settings demands a discretization of the infinite-dimensional space $\Ltwobdtwo$. As such, we replace the analytic field amplification \cref{eq:amplificationinf} with its discrete counterpart, which we still denote by $\phi$ with an abuse of notation:
\begin{equation*}
    \phi(k)=\norm{\uAk^{-1}\uM}_{\textup{op},\uM}:=\sup_{\vv\in\bR^{2\Nh}\setminus\{0\}}\norm{\uM^{1/2}\uAk^{-1}\uM\vv}_2\Big/\norm{\uM^{1/2}\vv}_2,
\end{equation*}
with $\norm{\cdot}_2$ denoting the Euclidean norm and $\uM$ denoting the mass matrix, cf.~\cref{sec:bemmat}. Recall that we are casting the discrete problem as described in \cref{rem:bemwithoutB}, thus justifying the right-multiplication of the resolvent operator by the mass matrix. By the change of variable $\vv\mapsto\uM^{-1/2}\vv$, we conclude that
\begin{equation}\label{eq:amplificationsvd}
    \phi(k)=\normsp{\uCk},\quad\text{given }\uCk=\uM^{1/2}\uAk^{-1}\uM^{1/2},
\end{equation}
with $\normsp{\cdot}$ representing the (Euclidean) spectral norm, i.e., the largest singular value of a matrix.

We note that, in practice, it is sometimes common to use a Cholesky factor of $\uM$ rather than its square root: given $\uM=\uRz^*\uRz$, one may define $\uCk=\uRz\uAk^{-1}\uRz^*$. Whether this is done or not is irrelevant here.

\begin{remark}\label{rem:amplificationwithoutB}
    If it is \emph{not} possible to ignore the operator $\calB{k}$ and the matrix $\uBk$, cf.~\cref{rem:RHSplanewave,rem:bemwithoutB}, one can define $\phi$ incorporating it: $\uCk=\uM^{1/2}\uAk^{-1}\uBk\uM^{-1/2}$.
    Since $\uBk$ is smooth with respect to $k$, our discussion below applies also in this case.
\end{remark}

By meromorphy of $\uAk^{-1}$, which follows from \cref{th:keldysh}, the matrix $\uCk$ is also meromorphic over $\bC$ with respect to $k$, with poles coinciding with those of $\uAk^{-1}$. As a consequence, $\phi$, by continuity of the spectral-norm operator, remains bounded over all compact subsets of $\bC\setminus\Lambda$, with $\Lambda\subset\bC$ being the set of all wavenumbers $k$ such that $\uAk$ is singular. However, this stops being true if one moves toward an eigenvalue, as the following proposition states.

\begin{proposition}\label{prop:asympt}
    Let $\lambda\in\bC$ be such that $\uA{\lambda}$ is singular, i.e., $\lambda\in\Lambda$. Let $m=m(\lambda)\in\bN$ be the length of the longest chain of generalized eigenvectors associated to $\lambda$, cf.~\cref{rem:jordanchain}. As $k\to\lambda$, there exist $\psi_\lambda>0$, depending only on the mass matrix $\uM$ and on the spectral projector $\textup{P}_{1,m}$, cf.~\cref{th:keldysh}, such that
    \begin{equation}\label{eq:asympt}
        \phi(k)=\psi_{\lambda}\abs{k-\lambda}^{-m}+\bigO{\abs{k-\lambda}^{1-m}}.
    \end{equation}
\end{proposition}
\begin{proof}
    Let $\uCk$ be defined as in \cref{eq:amplificationsvd}. \cref{th:keldysh} can be applied to $\uAk$ near $\lambda$. By \cref{th:keldysh}, we may choose a neighborhood $K$ small enough that $\lambda$ is the only pole of $\uAk^{-1}$ inside it. This leads to
    \begin{equation*}
        \uCk=\uM^{1/2}\uAk^{-1}\uM^{1/2}=\sum_{j=1}^m\frac{\uM^{1/2}\textup{P}_{1,j}\uM^{1/2}}{(k-\lambda)^j}+\uM^{1/2}\uHk\uM^{1/2},
    \end{equation*}
    which is valid over the punctured $K\setminus\{\lambda\}$. By the triangular inequality, we obtain
    \begin{align}
        \phi(k)=\normsp{\uCk}\leq&\normsp{\frac{\uM^{1/2}\textup{P}_{1,m}\uM^{1/2}}{(k-\lambda)^m}}+\normsp{\sum_{j=1}^{m-1}\frac{\uM^{1/2}\textup{P}_{1,j}\uM^{1/2}}{(k-\lambda)^j}+\uM^{1/2}\uHk\uM^{1/2}}\nonumber\\
        =&\frac{\normsp{\uM^{1/2}\textup{P}_{1,m}\uM^{1/2}}}{\abs{k-\lambda}^m}+\bigO{\abs{k-\lambda}^{1-m}},\label{eq:asymptC}
    \end{align}
    with the last step following from $\uHk=\bigO{1}$ by holomorphy of $\uHz$ over $K$. The claim then follows by setting $\psi_{\lambda}=\normsp{\uM^{1/2}\textup{P}_{1,m}\uM^{1/2}}$, which is strictly positive by \cref{th:keldysh}.
\end{proof}

\begin{remark}\label{rem:singlepole}
    Identifying the length $m(\lambda)$ of the longest chain of generalized eigenvectors associated to an eigenvalue $\lambda$ is generally a difficult problem. To simplify our work, we will assume that no non-trivial solution exists, so that all poles $\lambda$ have order $m(\lambda)=1$:
    \begin{equation}\label{eq:asymptsimp}
        \phi(k)=\frac{\psi_{\lambda}}{\abs{k-\lambda}}+\bigO{1}=\frac{\normsp{\uM^{1/2}\textup{P}_{1,1}\uM^{1/2}}}{\abs{k-\lambda}}+\bigO{1}\quad\textup{for }k\sim\lambda\in\Lambda.
    \end{equation}
    The above assumption is typical in practical settings, since, even if defective poles ($m(\lambda)>1$) exist, general perturbations of the scatterer $\Omega$ make the poles (semi-)simple again, recovering $m(\lambda)=1$ \cite{xiong_generic_2022}. From a ``numerical'' perspective, this means that round-off noise almost surely makes all poles' orders 1.
\end{remark}

More generally speaking, in cases where the spectrum $\Lambda$ includes defective eigenvalues ($m(\lambda)>1$), one may want to modify our upcoming techniques to account for the proper exponent $-m$ in \cref{eq:asympt}. Still, in all our numerical experiments, we have found no evidence to discard the assumption that $m(\lambda)=1$ for all $\lambda$. This fact, together with the intrinsic difficulties related to numerically dealing with defective eigenvalues (notably, identifying $m(\lambda)$), behooves us to leave the ``defective case'' as future research.

\begin{remark}\label{rem:nlevp}
    It is crucial to note that \cref{prop:asympt} provides a link between the field amplification function $\phi$ defined in \cref{eq:amplificationsvd} and the nonlinear eigenvalue problem \cref{eq:bdintformevp}. The singularities $\lambda_{i}$ of $\phi$, whose existence is proven by \cref{th:keldysh}, are located at all solutions to eigenproblem \cref{eq:bdintformevp}. However, solving the nonlinear eigenvalue problem is \emph{not enough} to characterize the field amplification. This is because of the additional smooth term $\uHk$, which affects $\phi$ but is invisible to the eigenvalue problem.
\end{remark}

The above results allow us to conclude that, although $\phi$ is bounded over any closed interval $K\subset\bR\setminus\{0\}$, if the strip of holomorphy of $\uCk$ is too narrow, cf.~\cref{eq:strip}, then almost-resonances will appear, in the form of ``spikes'' in $\phi$, cf.~\cref{eq:asympt,eq:asymptsimp}. This is confirmed by our numerical examples in \cref{sec:numexp}.

Now, recall that, in this work, we face the task of approximating $\phi$ over the frequency range $K$. To ease the approximation effort, it is crucial to identify the set of eigenvalues that limit the strip of holomorphy of $\uCk$, which we denote by $\Lambda_K$. Here we assume that $\Lambda_K$ contains the $M_K$ eigenvalues that lie closest to $K$, with $M_K$ a large enough number, roughly corresponding to how many local maxima (the above-mentioned ``spikes'') $\phi$ has in $K$.

By identifying the elements of $\Lambda_K$ and understanding their effect on $\phi$, one may circumvent some of the main issues due to the limited regularity of $\uCk$. In the next section, we describe how one can rely on rational approximation to achieve this, leveraging the meromorphic structure of the problem.

\section{Standard surrogate modeling}\label{sec:approx}
Evaluating the field amplification $\phi(k)$ is costly whenever the target scattering problem is complex enough to require a fine BEM discretization. Specifically, a direct computation of $\phi(k)$ by \cref{eq:amplificationsvd} generally follows the steps described in \cref{algo:direct}. Note that evaluating $\{\phi(k_i)\}_{i=1}^{\NK}$, with $\NK$ being the number of target wavenumbers, has computational cost $\bigO{\Nh^3\NK}$.

\begin{algorithm}[ht]
	\caption{Direct evaluation of $\{\phi(k_i)\}_{i=1}^{\NK}$}
	\begin{algorithmic}\label{algo:direct}
        \STATE{Assemble $\uM$ and compute $\uM^{1/2}$\hfill}\COMMENT{cost $\bigO{\Nh^3}$}
        \FOR{$k\in\{k_i\}_{i=1}^{\NK}$}
    		\STATE{Assemble the BE matrix $\uAk$\hfill}\COMMENT{cost $\bigO{\Nh^2}$}
    		\STATE{Compute $\uCk=\uM^{1/2}\uAk^{-1}\uM^{1/2}$\hfill}\COMMENT{cost $\bigO{\Nh^3}$}
            \STATE{Compute $\phi(k)$ by \cref{eq:amplificationsvd}\hfill}\COMMENT{cost $\bigO{\Nh^3}$}
            \LINECOMMENT{the cubic cost of the SVD may be lowered to quadratic-log by power iteration, cf.~\cref{sec:sub:sketch}}
        \ENDFOR
        \RETURN{$\{\phi(k_i)\}_{i=1}^{\NK}$}
	\end{algorithmic}
\end{algorithm}

We remark that this procedure is \emph{embarassingly parallel} with respect to $k$, i.e., the computations at each $k_i$ are independent from one another. As such, different $i$'s can be handled simultaneously by different computers, or by different cores in the context of parallel computing, without any communication overhead except for the final synthesis of all $\NK$ scalar results. Regardless, the cost of computing all field amplifications may still be too high even in a parallel setting, especially when $\Nh$ is large.

As a way to alleviate the computational burden, we employ \emph{surrogate modeling}. Specifically, we first construct a surrogate $\widetilde{\phi}:K\to\bR$, achieving good approximation quality: $\widetilde{\phi}\approx\phi$ over $K$. Then we replace the $\NK$ expensive evaluations of $\phi$ within \cref{algo:direct} with cheap evaluations of the surrogate $\widetilde\phi$.

\subsection{Rational approximation}\label{sec:sub:rat}
In order to justify rational approximation for field amplification, we take \cref{th:keldysh} as starting point. Therein, it is shown that the resolvent $\uAk^{-1}$ is a matrix-valued meromorphic function of the wavenumber $k$. As such, it is possible to \emph{efficiently} approximate it \emph{to arbitrary accuracy} via \emph{matrix-valued rational functions}, which, in this context, we define as the ratio of a matrix-valued polynomial and a scalar polynomial. For instance, using the \emph{barycentric} interpolatory rational form \cite{aaa}, we may approximate $\uCk$ as
\begin{equation}\label{eq:barymatrix}
    \uCk\approx\uLk:=\sum_{i=1}^{\nK}\frac{w_i\uC{\overline k_i}}{k-\overline k_i}\bigg/\sum_{i=1}^{\nK}\frac{w_i}{k-\overline k_i}.
\end{equation}
Barycentric rational interpolation has been extremely popular in recent years, thanks to its beneficial stability properties even for scattered (non-structured) sets of interpolation points $\{\overline k_i\}_{i=1}^{\nK}\subset\bC$, which makes it superior to many comparable approximation techniques. Note how the rational approximation consists of a linear combination of samples $\{\uC{\overline k_i}\}_{i=1}^{\nK}$ at arbitrary (distinct) locations $\{\overline k_i\}_{i=1}^{\nK}$, weighted by \emph{scalar}-valued rational functions
\begin{equation*}
    \left\{k\mapsto\frac{w_i}{k-\overline k_i}\bigg/\sum_{i'=1}^{\nK}\frac{w_{i'}}{k-\overline k_{i'}}\right\}_{i=1}^{\nK}.
\end{equation*}
One commonly refers to the scalar numbers $\{w_i\}_{i=1}^{\nK}\subset\bC$, which represent the degrees of freedom of the rational approximation, as \emph{barycentric weights}. Depending on the specific rational-approximation scheme, these are determined by solving a suitable data-fit or optimization problem.

Thanks to their meromorphic structure with respect to $k$, scattering problems have often been effectively dealt with by rational approximation in the surrogate-modeling community. See, e.g., \cite{hanke12,bonizzoni20,pradovera22,bruno24}. Even the extremely general ``reduced basis'' method \cite{chen10} (which fundamentally has no links to rational functions) sometimes yields a rational surrogate when applied to the Helmholtz equation.

\begin{algorithm}[ht]
	\caption{Approximation of $\{\phi(k_i)\}_{i=1}^{\NK}$ by rational surrogate modeling}
	\begin{algorithmic}\label{algo:rat}
        \STATE{Compute $\uM^{1/2}$\hfill}\COMMENT{cost $\bigO{\Nh^3}$}
        \STATE{Select $\nK$ sample points $\{\overline k_i\}_{i=1}^{\nK}\subset\bC$, e.g., adaptively \cite{pradovera22,pradovera23,aaacontinuum}}
        \FOR{$\overline k\in\{\overline k_i\}_{i=1}^{\nK}$}
    		\STATE{Compute sample $\uC{\overline k}=\uM^{1/2}\uA{\overline k}^{-1}\uM^{1/2}$\hfill}\COMMENT{cost $\bigO{\Nh^3}$}
        \ENDFOR
        \STATE{Get the barycentric weights $\{w_i\}_{i=1}^{\nK}$ by solving a data-fit problem\hfill}\COMMENT{cost $\bigO{\Nh^2\nK^3}$, cf.~\cref{rem:baryweights}}
        \STATE{Assemble the rational surrogate $\uLz$ as in \cref{eq:barymatrix}}
        \FOR{$k\in\{k_i\}_{i=1}^{\NK}$}
    		\STATE{Evaluate $\uLk$\hfill}\COMMENT{cost $\bigO{\Nh^2\nK}$}
            \STATE{Compute $\widetilde\phi(k)$ as largest singular value of $\uLk$\hfill}\COMMENT{cost $\bigO{\Nh^3}$}
        \ENDFOR
        \RETURN{$\{\widetilde\phi(k_i)\}_{i=1}^{\NK}$}
	\end{algorithmic}
\end{algorithm}

A computational pipeline for rational-interpolation-based approximation of $\phi$ may be as in \cref{algo:rat}. Therein, we assume that a strategy for selecting the number and locations of the sample points $\{\overline k_i\}_{i=1}^{\nK}$ is available. For instance, one may use strategies for adaptive sampling, with the algorithm automatically selecting the sample locations in a (quasi-)optimal way so as to attain a desired approximation accuracy. These have been quite popular in recent years \cite{aaa,pradovera22,pradovera23,bruno24,aaacontinuum}. In such approaches, the number $\nK$ of samples is automatically selected in a data-driven way, based on the observed features of the target $\uCk$.

Standard results in rational approximation \cite{walsh60,baker96} suggest that $\nK>M_K$, with $M_K$ being the number of poles of $\uCk$ close to $K$, cf.~\cref{sec:sub:field}, is necessary for an acceptable approximation accuracy. At the same time, one expects an adaptive sampling routine to be parsimonious, i.e., that not so many expensive samples are taken. We will merge these properties into the condition $\nK\sim M_K$.

\begin{remark}\label{rem:baryweights}
    In \cref{algo:rat}, one needs to compute the barycentric weights $\{w_i\}_{i=1}^{\nK}$. How this is done depends on which rational-approximation method was chosen. Generally, this step consists of solving a suitable optimization problem. For instance, in set-valued AAA \cite{lietaert22} and minimal rational interpolation \cite{bonizzoni23}, one finds the barycentric weights by solving a certain least-squares problem involving a Loewner-like matrix of size $\bigO{\Nh^2\nK}\times \nK$. Note, in particular, how the number of rows of the Loewner matrix scales with $\Nh^2$ as a consequence of the unraveling of the samples from $(2\Nh)\times(2\Nh)$ matrices into vectors of size $4\Nh^2$. The cost of solving this least-squares problem is $\bigO{\Nh^2\nK^3}$, the scaling which is reported in \cref{algo:rat} as reference. More advanced strategies based on QR factorization may reduce this cost to $\bigO{\Nh^2\nK^2}$ \cite{bonizzoni23}.
\end{remark}

In terms of algorithmic efficiency, it is crucial to assume that $\nK\ll \NK$. Indeed, if this is not the case, a surrogate model approximating $\phi$ over the target wavenumbers $\{k_i\}_{i=1}^{\NK}$ is just as complex as $\phi$ itself. As such, rational approximation loses its effectiveness, and a direct evaluation of $\{\phi(k_i)\}_{i=1}^{\NK}$ as in \cref{algo:direct} is recommended instead.

Moreover, for simplicity, from now on we will assume that $\nK\ll \Nh$, i.e., $M_K\ll \Nh$. In practical settings, this is a reasonable assumption if the ``spatial'' complexity of the scattering problem dominates the ``spectral'' complexity of the problem.

\begin{remark}\label{rem:spectralassumption}
    Our assumption that $M_K\ll \Nh$ is often true in practice, especially in cases where a high spatial resolution is needed for good accuracy of the BEM discretization, e.g., if the scatterer has a complicated geometry. Still, $M_K\gtrsim \Nh$ may happen if the problem has an extremely complex spectrum, cf.~the numerical results in \cref{sec:numexp}. In that case, rational approximation is less beneficial at reducing the complexity of computing $\phi$, since $\nK>M_K\gtrsim \Nh$. Specifically, sampling becomes the main bottleneck of any method for approximating $\phi$ in this case.
\end{remark}

Under the above assumptions, \cref{algo:rat} costs $\bigO{\Nh^3\nK+\Nh^2\nK^3+\Nh^3\NK}$, with the surrogate-evaluation cost being $c_{\textup{surrogate}}=\bigO{\Nh^3}$, which is still the same as in \cref{algo:direct}. The bottleneck in the surrogate evaluation is the computation of the largest singular value of $\uLk$. To achieve that, one may replace the cubic-cost SVD with a small number $q=\bigO{\log \Nh}$ of power-method iterations, achieving $\bigO{\Nh^2\log \Nh}$ cost, cf.~\cref{sec:sub:sketch}. In this way, one can decrease the surrogate-evaluation cost to $c_{\textup{surrogate}}=\bigO{\Nh^2(\nK+\log \Nh)}$ and the total complexity to $\bigO{\Nh^3\nK+\Nh^2\nK^3+\Nh^2(\nK+\log \Nh)\NK)}$. This is generally an improvement over the direct evaluation of $\phi$, although the algorithm has a super-quadratic-in-$\Nh$ surrogate-evaluation cost. An at-least-quadratic cost is unavoidable here, since we are approximating all entries of the $(2\Nh)\times(2\Nh)$ matrix $\uCk$. This can be improved, as we proceed to describe.

\subsection{Compression by randomized power method}\label{sec:sub:sketch}

The randomized power method (RPM) falls within the family of \emph{randomized} methods in computational linear algebra, where randomness is leveraged to reduce the computational cost of linear-algebra operations, e.g., matrix compression \cite{halko11,woolfe08} and approximating vector norms \cite{dixon83,martinsson20}. In randomized linear algebra, a large number of operations on deterministic vectors is replaced with a smaller number of operations on random vectors, introducing a small but nonzero probability of failure in the process.

The RPM can be applied in the specific case of spectral-norm approximation. 
Consider a matrix $\uCz\in\bC^{N\times N}$, a random vector $\vb\in\bC^N$, e.g., with i.i.d.~normal Gaussian entries, and an arbitrary number of iterations $q\geq0$. We can approximate
\begin{equation}\label{eq:sketch}
    \normsp{\uCz}\approx\norm{\uCz}_{\textup{sp},\textup{RPM},q}:=\frac{\norm{\uCz(\uCz^*\uCz)^q\vb}}{\norm{(\uCz^*\uCz)^q\vb}}.
\end{equation}
This is simply the square root of the estimate for the largest eigenvalue of $\uCz^*\uCz$ that one gets after $q$ iterations of the well-known power method with starting vector $\vb$.

Standard results in basic computational linear algebra show the convergence of $\norm{\uCz}_{\textup{sp},\textup{RPM},q}$ to $\normsp{\uCz}$ as $q\to\infty$. 
%
In general, a minimum number of iterations $q\sim\log N$ (the so-called ``burn-in period'') is needed to guarantee a good approximation accuracy \cite{simchowitz17,martinsson20}. 
For simplicity, from here onward, we use the above result to prescribe the scaling $q\sim\log N$, so that the cost of computing $\norm{\uCz}_{\textup{sp},\textup{RPM},q}$ is $\bigO{N^2\log N}$.


As previously suggested, a straightforward application of RPM can speed up the spectral norm evaluations in \cref{algo:direct,algo:rat}, since its $\bigO{\Nh^2\log \Nh}$ cost is an improvement over the $\bigO{\Nh^3}$ cost of SVD.

However, we can employ RPM in a different way, leading to further computational savings. The main idea, rather than building a rational approximation of a matrix and then applying the RPM to it, is to apply the RPM first and then build rational approximations of the two vectors appearing in the RPM approximation, cf.~\cref{eq:sketch}. More rigorously, one may construct two rational functions $\vell_{0}$ and $\vell_{1}$, approximating the two $\bC^{2\Nh}$-valued functions
\begin{equation}\label{eq:barysketched}
    \uCk^p\left(\uCk^*\uCk\right)^q\vbh\approx\vell_{p}(k):=\sum_{i=1}^{\hatnK}\frac{w_{p,i}\uC{\overline k_i}^p\left(\uC{\overline k_i}^*\uC{\overline k_i}\right)^q\vbh}{k-\overline k_i}\bigg/\sum_{i=1}^{\hatnK}\frac{w_{p,i}}{k-\overline k_i},
\end{equation}
for $p\in\{0,1\}$. This leads to the procedure outlined in \cref{algo:sketch}.

\begin{algorithm}[ht]
	\caption{Approximation of $\{\phi(k_i)\}_{i=1}^{\NK}$ by RPM and rational surrogate modeling}
	\begin{algorithmic}\label{algo:sketch}
        \STATE{Compute $\uM^{1/2}$\hfill}\COMMENT{cost $\bigO{\Nh^3}$}
        \STATE{Draw the random vector $\vbh\in\bC^{2\Nh}$ with i.i.d.~Gaussian entries}
        \STATE{Set an appropriate number of iterations $q$, e.g., $q\sim\log \Nh$}
        \STATE{Select $\hatnK$ sample points $\{\overline k_i\}_{i=1}^{\hatnK}\subset\bC$, e.g., adaptively \cite{pradovera22,pradovera23,aaacontinuum}}
        \FOR{$\overline k\in\{\overline k_i\}_{i=1}^{\hatnK}$}
    		\STATE{Compute $\uC{\overline k}=\uM^{1/2}\uA{\overline k}^{-1}\uM^{1/2}$\hfill}\COMMENT{cost $\bigO{\Nh^3}$}
            \STATE{Get samples $\left(\uC{\overline k}^*\uC{\overline k}\right)^q\vbh$ and $\uC{\overline k}\left(\uC{\overline k}^*\uC{\overline k}\right)^q\vbh$\hfill}\COMMENT{cost $\bigO{q\Nh^2}$}
        \ENDFOR
        \STATE{Get the weights $\{w_{p,i}\}_{p=0,i=1}^{1,\hatnK}$ by solving two data-fit problems\hfill}\COMMENT{cost $\bigO{\Nh\hatnK^3}$, cf.~\cref{rem:sketchedweights}}
        \STATE{Assemble the rational surrogates $\{k\mapsto\vell_{p}(k)\}_{p=0}^1$ as in \cref{eq:barysketched}}
        \FOR{$k\in\{k_i\}_{i=1}^{\NK}$}
            \STATE{Evaluate surrogates $\{\vell_{p}(k)\}_{p=0}^{1}$ and compute $\widetilde\phi(k)=\norm{\vell_{1}(k)}/\norm{\vell_{0}(k)}$\hfill}\COMMENT{cost $\bigO{\hatnK^2}$, cf.~\cref{rem:surrogatesketch}}
        \ENDFOR
        \RETURN{$\{\widetilde\phi(k_i)\}_{i=1}^{\NK}$}
	\end{algorithmic}
\end{algorithm}

\begin{remark}\label{rem:sketchedweights}
    Thanks to the reduced size of each rational function's output space, which is now $2\Nh$ (a BE vector) rather than $4\Nh^2$ (a BE matrix), computing the barycentric weights has cost $\bigO{\Nh\hatnK^3}$, cf.~\cref{rem:baryweights}.
\end{remark}

\begin{remark}\label{rem:surrogatesketch}
    A naive evaluation of $\widetilde\phi(k)=\norm{\vell_{1}(k)}/\norm{\vell_{0}(k)}$ at each target point costs $c_{\textup{surrogate}}=\bigO{\Nh\hatnK}$, since this is the complexity associated to evaluating the two size-$(2\Nh)$ rational functions $\{\vell_{p}(k)\}_{p=0}^1$. However, this cost can be made $\Nh$-independent by pre-computing certain inner products of samples. Specifically, given
    \begin{equation*}
        \alpha_{p,i,j}:=\left(\uC{\overline k_i}^p\left(\uC{\overline k_i}^*\uC{\overline k_i}\right)^q\vbh\right)^*\left(\uC{\overline k_j}^p\left(\uC{\overline k_j}^*\uC{\overline k_j}\right)^q\vbh\right)
    \end{equation*}
    for $p\in\{0,1\}$ and $i,j=1,\ldots,\hatnK$, one has the expansion
    \begin{equation*}
        \widetilde\phi(k)^2=\frac{\norm{\vell_{1}(k)}^2}{\norm{\vell_{0}(k)}^2}=\frac{\sum_{i,j=1}^{\hatnK}\frac{w_{1,i}^*w_{1,j}\alpha_{1,i,j}}{(k-\overline k_i)(k-\overline k_j)^*}}{\sum_{i,j=1}^{\hatnK}\frac{w_{1,i}^*w_{1,j}}{(k-\overline k_i)(k-\overline k_j)^*}}\cdot\frac{\sum_{i,j=1}^{\hatnK}\frac{w_{0,i}^*w_{0,j}}{(k-\overline k_i)(k-\overline k_j)^*}}{\sum_{i,j=1}^{\hatnK}\frac{w_{0,i}^*w_{0,j}\alpha_{0,i,j}}{(k-\overline k_i)(k-\overline k_j)^*}},
    \end{equation*}
    whose evaluation has cost $c_{\textup{surrogate}}=\bigO{\hatnK^2}$.
\end{remark}

Assume that the number $\hatnK$ of samples needed to build the two rational surrogates\footnote{For simplicity of notation, we are assuming that the same sample points are used for both rational surrogates $\vell_{0}(k)$ and $\vell_{1}(k)$. In practice, we even allow the two surrogates to be built using different numbers of samples at different locations.} satisfies $\hatnK\ll \Nh$, cf.~\cref{rem:spectralassumption}. Following the above remarks, one can make the complexity of \cref{algo:sketch} equal $\bigO{\Nh^3\hatnK+\hatnK^2\NK}$, with $c_{\textup{surrogate}}=\bigO{\hatnK^2}$. As such, we can see that \cref{algo:sketch} is an improvement over \cref{algo:rat} as long as the number of samples needed in the two algorithms is similar ($\nK\approx\hatnK$).

Unfortunately, this is not the case, since the two algorithms apply rational approximation to different functions. Specifically, \cref{algo:rat} has the ``easy'' task of approximating the \emph{meromorphic} function $k\mapsto\uCk$, cf.~\cref{th:keldysh}. On the other hand, \cref{algo:sketch} deals with the approximation of\footnote{For the sake of conciseness, we are assuming simple poles.}
\begin{equation}\label{eq:nonmero}
    \uCk^p\left(\uCk^*\uCk\right)^q\vbh
    =\sum_{i=1}^{M_K}\frac{\big(\uM^{1/2}\textup{P}_{i,1}\uM^{1/2}\big)^p\big(\uM^{1/2}\textup{P}_{i,1}^*\uM\textup{P}_{i,1}\uM^{1/2}\big)^q\vbh}{(k-\lambda_{i})^p\abs{k-\lambda_{i}}^{2q}}+\bigO{\abs{k-\lambda_{i}}^{p+2q-1}},
\end{equation}
which, for $q>0$, \emph{is not meromorphic anymore}, since complex conjugation does not preserve holomorphy. The lack of meromorphy impacts the \emph{effectiveness} of rational approximation: the number $\hatnK$ of samples needed to attain an ``acceptable'' approximation accuracy is generally much larger than $(2qM_K)$, which a meromorphic function of a similar form would require. Our numerical experiments in \cref{sec:numexp} confirm this.

\begin{remark}\label{rem:nonmero}
    Roughly speaking, \cref{eq:nonmero} shows that the effectiveness of rational approximation within \cref{algo:sketch} decreases as $q$ increases. If no iterations are performed ($q=0$), only the \emph{meromorphic} function
    \begin{equation*}
        k\mapsto\uCk\vbh=\sum_{i=1}^{M_K}\frac{\uM^{1/2}\textup{P}_{i,1}\uM^{1/2}\vbh}{k-\lambda_{i}}+\uM^{1/2}\uHk\uM^{1/2}\vbh,
    \end{equation*}
    with $\vbh$ a random vector, needs to be approximated within \cref{algo:sketch}. This can be done with the same ease as approximating $\uCk$ directly, cf.~\cref{algo:rat}.
\end{remark}
    
In summary, there are no theoretical grounds for concluding that \cref{algo:sketch} might actually perform better than \cref{algo:rat}, in terms of surrogate assembly and evaluation.

\section{Hybrid surrogate modeling}\label{sec:hybrid}
In the previous sections, we have discussed two surrogate-modeling-based methods for approximating $\phi$, which follow rigorous function-approximation and computational-linear-algebra theory. In this section, we describe an alternative approach, which approximates $\phi$ following more heuristic arguments. Thanks to some drastic simplifications, the upcoming ``hybrid'' method, described in \cref{sec:sub:hybrid} below, is extremely efficient. Indeed, it builds a surrogate model using roughly as few samples as \cref{algo:rat} (hence, at roughly the same cost), with an extremely low surrogate-evaluation cost $c_{\textup{surrogate}}=\bigO{\nK}$.

However, the heuristic step that leads to an enhanced efficiency also causes a decrease in accuracy. Specifically, although our upcoming method provides an approximation that converges to the exact $\phi$, cf.~\cref{prop:conv}, such convergence is slow. For this reason, it is more useful to think of our hybrid approach as building a ``rough'' approximation of $\phi$. What still justifies our interest in such method is that (i) the approximation retains good accuracy near almost-resonating wavenumbers 
and (ii) a higher accuracy can still be achieved if one is willing to pay a higher (but still competitive) cost for building the surrogate.

In order to derive our hybrid algorithm, we start from \cref{rem:singlepole} and apply \cref{eq:asymptsimp} near all troublesome poles within $\Lambda_K$. Ignoring bounded $\bigO{1}$-terms, we obtain
\begin{equation}\label{eq:asmatrix}
    \phi(k)=\sup_{\norm{\vv}_2=1}\norm{\uCk\vv}_2\approx\sup_{\norm{\vv}_2=1}\norm{\sum_{i=1}^{M_K}\frac{\uM^{1/2}\textup{P}_{i,1}\uM^{1/2}\vv}{k-\lambda_{i}}}_2,
\end{equation}
with $\{\lambda_{i}\}_{i=1}^{M_K}$ being an enumeration of $\Lambda_K$. Note, in particular, that, since the eigenspaces are not pairwise orthogonal \cite{betcke14}, we cannot simplify the above expression through the Pythagorean theorem.

Moreover, we observe that, in order to use \cref{eq:asmatrix} directly, we need to compute the (oblique) spectral projectors $\textup{P}_{i,1}$. This not only is computationally expensive, but also requires taking samples of $\uC{\cdot}$ at non-purely-real wavenumbers, which is impossible in our framework. As such, we propose to avoid handling such projectors, by looking at a more manageable \emph{scalar} expression, obtained by applying the triangular inequality:
\begin{equation}\label{eq:assum}
    \phi(k)\lessapprox\widetilde\phi_{\textup{sum},0}(k):=\sum_{i=1}^{M_K}\sup_{\norm{\vv}_2=1}\norm{\frac{\uM^{1/2}\textup{P}_{i,1}\uM^{1/2}\vv}{k-\lambda_{i}}}_2=\sum_{\lambda\in\Lambda_K}\frac{\psi_{\lambda}}{\abs{k-\lambda}},
\end{equation}
recalling the definition $\psi_{\lambda_i}=\normsp{\uM^{1/2}\textup{P}_{i,1}\uM^{1/2}}$.

At this point, one might note that \cref{eq:assum} does not avoid the computation of the spectral projectors $\textup{P}_{i,1}$, since they appear in the definitions of the scalar quantities $\psi_{\lambda_{i}}$. Moreover, the approximation relies on the unknown eigenvalues $\Lambda_K$. We now describe how these issues can be circumvented, through a combination of rational approximation and linear algebra, by approximating the set $\Lambda_K$ (\cref{sec:hybrid:poles}) and by fitting the expansion coefficients $\psi_{\lambda_{i}}$ from data (\cref{sec:sub:fit}).

\subsection{Identifying the poles}\label{sec:hybrid:poles}
Approximating the eigenvalues of a nonlinear eigenproblem is a common task in computational linear algebra. Solvers based on contour integration \cite{beyn12} or (local) linearizations \cite{guttel17} are quite popular. Unfortunately, linearizations are not very effective for our problem, due to the extremely nonlinear (in $k$) nature of the BE matrices. Moreover, contour integration is unavailable since, by assumption, we cannot sample the BE matrices at non-real wavenumbers. As a feasible alternative, we use rational approximation once more, similarly to \cite{pradoveraEVP23}. Specifically, we start by computing a rational approximation of either the BE matrix $\uLk\approx\uCk$, a one-sided sketch $\vell(k)\approx\uCk\vbh$, with $\vbh$ being a random vector, or a two-sided sketch $\zeta(k)\approx\mathbf{b}_{1}^*\uCk\mathbf{b}_{2}$, with $\mathbf{b}_{1}$ and $\mathbf{b}_{2}$ being random vectors. Then we find an approximation of $\Lambda_K$ by computing the set of poles of $\uLk$, $\vell(k)$, or $\zeta(k)$. Crucially, such pole-finding requires non-real computations of the rational approximation only, without involving any evaluation of the BE matrices at non-real wavenumbers.

In summary, given $\nK$ samples, in this way we can approximate $\Lambda_K$ with complexity $\bigO{\Nh^3\nK+\Nh^\eta \nK^3}$, with $\eta\in\{0,1,2\}$ depending on what kind of sketching is used, if any. Two-sided sketching yields the most efficient algorithm ($\eta=0$) and is the alternative chosen, e.g., in \cite{bruno24}. However, in our tests, we use one-sided sketching ($\eta=1$) because of two reasons:
\begin{itemize}
    \item it has lower cost than the non-sketched option ($\eta=2$) at (empirically) no accuracy reduction;
    \item the vector-valued nature of $k\mapsto\vell(k)\in\bC^{2\Nh}$ affords us the convenience of using the \emph{adaptive} rational-approximation tools from \cite{bonizzoni23,pradovera23}, which cannot be applied to the scalar-valued $k\mapsto\zeta(k)$.
\end{itemize}

\begin{remark}\label{rem:spurious}
    In using rational approximation for pole-finding, one can fall in a common pitfall, related to the appearance of so-called ``spurious poles''. Spurious poles are poles of the rational approximation that do not correspond to any pole of the original function. These may appear even in exact arithmetic \cite[Chapter 6.1]{baker96}, and are automatically placed by the rational-approximation algorithm at locations that are ``convenient'' to improve the approximation quality. As it so often happens, round-off noise exacerbates this issue in numerical settings. When identifying $\Lambda_K$ by the set of poles of a rational approximation, one would ideally want to filter out any spurious poles.
\end{remark}

Unfortunately, the spurious-pole filtering suggested by the above remark is one of the hardest tasks in rational approximation. Some (more or less) heuristic strategies for this exist \cite{aaa,bruno24}, flagging poles as spurious based on few additional samples of the target function. Unfortunately for our purposes, such additional samples must be placed near the poles. Due to the complex nature of the poles, cf.~\cref{sec:spectral}, this is impossible in our setting, where evaluations of the BE matrices are constrained to the real axis. There are also other techniques that flag poles as spurious if the corresponding \emph{approximate residue}, i.e., the residue in a Heaviside expansion of the rational surrogate, is small. However, we were not able to apply these ideas effectively in our experiments, since we have often found spurious residues with even larger magnitude than non-spurious ones. For these reasons, in our algorithm we choose to approximate $\Lambda_K$ either by the set of \emph{all} surrogate poles
\begin{equation*}
    \widetilde\Lambda_K^{\textup{all}}:=\left\{\lambda\in\bC:\sum_{i=1}^{\nK}\frac{w_i}{\lambda-\overline k_i}=0\right\},
\end{equation*}
cf.~\cref{eq:barymatrix}, or by the subset of all those surrogate poles that are closest to some $k\in\{k_i\}_{i=1}^{\NK}$, i.e.,
\begin{equation}\label{eq:polefilter}
    \widetilde\Lambda_K^{\textup{close}}:=\left\{\lambda\in\widetilde\Lambda_K^{\textup{all}}:\exists k\in\{k_i\}_{i=1}^{\NK}\textup{ such that }\abs{\lambda-k}\leq\abs{\lambda'-k}\forall\lambda'\in\widetilde\Lambda_K^{\textup{all}}\right\}.
\end{equation}
The rationale behind taking the subset $\widetilde\Lambda_K^{\textup{close}}\subset\widetilde\Lambda_K^{\textup{all}}$, rather than $\widetilde\Lambda_K^{\textup{all}}$ itself, is that the latter set often includes a certain amount of spurious poles, located \emph{at some distance} from $K$. Indeed, there are theoretical grounds for expecting at least some of the spurious poles to cluster on a relatively large Bernstein ellipse around the sampling interval $K$ \cite{trefethen23}. Numerical evidence for this, as well as a discussion of the effects that spurious poles have on the resulting approximation of $\phi$, are included in \cref{sec:numexp}.

\subsection{Assembling the approximate field amplification}\label{sec:sub:fit}
Now, assume that $\Lambda_K$ has been well approximated as $\widetilde\Lambda_K=\{\lambda_{i}\}_{i=1}^{M_K'}$ by the above rational-approximation-based strategy, where, with an abuse of notation, we are denoting the surrogate poles as $\lambda_{i}$. The number of approximate eigenvalues satisfies $M_K'\leq \nK-1$ and should ideally be close to $M_K$. We face the task of computing the coefficients $\psi_{\lambda}$ in the approximate expansion \cref{eq:assum}, with $\widetilde\Lambda_K$ replacing $\Lambda_K$. To this aim, we choose to use the following approach based on \emph{collocation}.

We start by observing that the coefficients $\{\psi_{\lambda}\}_{\lambda\in\widetilde\Lambda_K}$ can be interpreted as the degrees of freedom appearing in an approximation of $\phi$ that uses a linear combination of the ``basis functions'' $$\left\{\bR\ni k\mapsto\abs{k-\lambda}^{-1}=\left((k-\real\lambda)^2+(\imag\lambda)^2\right)^{-1/2}\in\bR\right\}_{\lambda\in\widetilde\Lambda_K}.$$ As long as the approximate poles are pairwise distinct and pairwise not complex conjugate, the above basis functions are linearly independent. In this setting, any set of $M_K'$ distinct \emph{collocation points} $\{\widetilde k_i\}_{i=1}^{M_K'}\subset\bR$ is ``unisolvent'' with this basis, in the sense that the data-fit problem
\begin{equation}\label{eq:assumfit}
    \textup{find }\{\widetilde\psi_{\lambda}\}_{\lambda\in\widetilde\Lambda_K}\textup{ such that }\phi(\widetilde k_i)=\sum_{\lambda\in\widetilde\Lambda_K}\frac{\widetilde\psi_{\lambda}}{\abs{\widetilde k_i-\lambda}}\quad\textup{for all }i=1,\ldots,M_K'
\end{equation}
admits a unique solution. As such, it is possible to find approximations $\widetilde\psi_{\lambda}\approx\psi_{\lambda}$ by sampling the exact $\phi$ at the $M_K'$ distinct points $\{\widetilde k_i\}_{i=1}^{M_K'}$, and then solving \cref{eq:assumfit}. In this way, one can sidestep the complicated projectors $\textup{P}_{i,1}$ from \cref{eq:assum}.

\begin{remark}
    Since the approximate poles are automatically computed as described in \cref{sec:hybrid:poles}, it may happen that two of them may coincide or be complex conjugate, thus making \cref{eq:assumfit} ill-posed due to repeated basis functions. This issue is easily solved by pre-processing $\widetilde\Lambda_K$, removing a copy of each repeated or complex-conjugate approximate pole. Note, in particular, that repeated poles might be a symptom of defects in the problem's spectrum, cf.~\cref{rem:singlepole} and ensuing paragraph.
\end{remark}

Leveraging common ideas from the radial-basis literature, one may improve both approximation accuracy and numerical stability by placing the collocation points at the locations where the basis functions attain their respective maxima. See, e.g., \cite{wendland04}. In our setting, this happens at $\widetilde k_i:=\real\lambda_{i}$ for $i=1,\ldots,M_K'$. If any real parts of the approximate poles repeat, one should nudge any repeated collocation points by a small (real) amount.

Also stemming from the radial-basis literature is the idea of enriching the radial basis with some additional basis functions, e.g., polynomials up to a certain (low) degree. To account for the increased basis size, one must either add linear constraints or include additional samples. Although the former approach is more efficient, we choose to follow the latter approach since the additional samples help us achieve a higher accuracy. We introduce $n_o\geq0$ additional samples\footnote{The subscript ``o'' stands for ``oversampling''.} $\{\widetilde k_{i+M_K'}\}_{i=1}^{n_o}\subset\bR$, e.g., at uniformly spaced locations over $K$. Then we fix an arbitrary basis $\{\psi_j\}_{j=1}^{n_b}$, with $0\leq n_b\leq n_o$, and we construct the approximation
\begin{equation}\label{eq:assumpoly}
    \widetilde\phi_{\textup{sum},n_b}(k)=\sum_{\lambda\in\widetilde\Lambda_K}\frac{\widetilde\psi_{\lambda}}{\abs{k-\lambda}}+\sum_{j=1}^{n_b}\alpha_j\psi_j(k).
\end{equation}
The expansion coefficients can be easily found by solving the data-fit problem
\begin{equation}\label{eq:assumfitpoly}
    \phi(\widetilde k_i)=\widetilde\phi_{\textup{sum},n_b}(\widetilde k_i)\quad\textup{for all }i=1,\ldots,M_K'+n_o,
\end{equation}
either exactly (if $n_b=n_o$) or in a least-squares sense (if $n_b<n_o$).

While the accuracy of the basic approximation $\widetilde\phi_{\textup{sum},0}$ hinges on the accuracy of the approximate expansion \cref{eq:assum}, the addition of extra terms can aid in overcoming this limitation. Indeed, since the coefficients of $\widetilde\phi_{\textup{sum},n_b}$ are found by fitting \emph{exact} values of $\phi$, the extra terms can reduce the ``modeling mismatch'' incurred in the scalar heuristic approximation \cref{eq:assum}. Based on this, we can conclude that $\widetilde\phi_{\textup{sum},n_b}$ \emph{converges} to the exact $\phi$ as $n_b\to\infty$, in the following sense.

\begin{proposition}\label{prop:conv}
    Let $K\subset\bR\setminus\{0\}$ be compact. Let $\{\psi_j\}_{j=1}^\infty$ be a sequence of functions $\psi_j:K\to\bR$, and assume that any continuous positive function can be approximated by the family $\{\psi_j\}_{j=1}^\infty$ to arbitrary precision in the $L^\infty(K)$-norm:
    \begin{equation}\label{eq:conv}
        \forall\mu\in C(K;\bR_{>0})\quad\lim_{n_b\to\infty}\inf_{\bm\alpha\in\bR^{n_b}}\sup_{k\in K}\abs{\mu(k)-\sum_{j=1}^{n_b}\alpha_j\psi_j(k)}=0.
    \end{equation}
    Then, if $\widetilde\Lambda_K\cap K=\emptyset$, one may choose the coefficients $\{\alpha_j\}_{j=1}^{n_b}$ in \cref{eq:assumpoly} so that $\widetilde\phi_{\textup{sum},n_b}$ converges to $\phi$ in the $L^\infty(K)$-norm as $n_b\to\infty$.
\end{proposition}
\begin{proof}
    As discussed in \cref{sec:spectral}, the spectrum of the scattering problem satisfies $\Lambda\subset\{0\}\cup\{z\in\bC,\imag z<0\}$ and has no finite accumulation points. As such, the interval $K$ has nonzero distance from $\Lambda$. Accordingly, the field amplification $\phi$ is bounded uniformly over $K$ by \cref{th:keldysh}. Recalling definition \cref{eq:amplificationsvd}, the continuity of $\phi$ can be easily deduced by combining the facts that (i) the matrix $\uAk^{-1}$ depends continuously on $k\in\bC\setminus\Lambda$, cf.~\cref{th:keldysh}, and (ii) the spectral norm is continuous as a $\bC^{(2\Nh)\times(2\Nh)}\to\bR_{\geq 0}$ function.

    If no surrogate pole $\lambda\in\widetilde\Lambda_K$ lies in $K$, the sum $\widetilde\phi_{\textup{sum},0}$, cf.~\cref{eq:assum}, is a continuous function. As such, the claim follows by applying \cref{eq:conv} to the difference $\mu=\phi-\widetilde\phi_{\textup{sum},0}$, which is also continuous over $K$.
\end{proof}

Note that the ``universal-approximation property'' \cref{eq:conv}, which is sufficient for convergence, is satisfied by many common hierarchical bases, e.g., polynomials of increasing degree, the Fourier basis, or the piecewise-linear finite element basis on nested meshes. Low-regularity bases, e.g., the last above-mentioned option, are recommended for higher numerical stability, due to the low smoothness of $\phi$. In addition, we remark that the condition $\widetilde\Lambda_K\cap K=\emptyset$ is generally satisfied in practice due to round-off noise. If this is not true, it suffices to ``nudge'' any real poles by a small imaginary shift $(-\varepsilon\textup{i})$, $0<\varepsilon\ll1$.

Moreover, it is important to note that the convergence predicted by \cref{prop:conv} will inevitably be slow, due to the low regularity of $\phi$, which is continuous but not differentiable, due to the usual presence of kinks, as showcased, e.g., in \cite{grubisic23} and in our numerical tests in \cref{sec:numexp}. Such kinks are a simple consequence of the non-differentiability of the spectral norm. As such, we expect a convergence of order at most $\bigO{n_b^{-1}}$.

\begin{remark}\label{rem:conv}
    \cref{prop:conv} may be adapted to show that a suitable superposition of the basis functions $\{\psi_j\}_{j=1}^{n_b}$ converges to $\phi$ as $n_b\to\infty$ at the same rate \emph{even without the additional term} $\widetilde\phi_{\textup{sum},0}$. The justification for including this latter term is providing a better ``starting point'' for the surrogate, which already includes a fairly good approximation of any almost-resonant behavior.
\end{remark}

There exists an alternative to approximation \cref{eq:assumpoly}. Rather than using the triangular inequality to obtain the \emph{sum} of the fractions in \cref{eq:assum}, we can simply take their \emph{maximum}, in some sense following \cref{prop:asympt}:
\begin{equation}\label{eq:asmax}
    \phi(k)\approx\widetilde\phi_{\max,n_b}(k):=\max_{\lambda\in\Lambda_K}\frac{\widetilde\psi_{\lambda}}{\abs{k-\lambda}}+\sum_{j=1}^{n_b}\alpha_j\psi_j(k).
\end{equation}
The numerators $\widetilde\psi_{\lambda}$ may be found, e.g., with a heuristic collocation-based approach, through a direct sampling of $\phi(\widetilde k)$ at $\widetilde k=\real\lambda$:
\begin{equation*}
    \widetilde\psi_{\lambda}:=\abs{\real\lambda-\lambda}\phi(\real\lambda)=\abs{\imag\lambda}\phi(\real\lambda)\quad\textup{for any }\lambda\in\Lambda_K.
\end{equation*}
By construction, this simple choice yields $\phi(\real\lambda)=\widetilde\phi_{\max,0}(\real\lambda)$ for any $\lambda\in\Lambda_K$, as long as the maximum in \cref{eq:asmax} is attained at $\lambda$, i.e., whenever no approximate pole ``hides'' any other approximate pole. In particular, note that the pole pre-processing outlined in \cref{eq:polefilter} greatly helps in preventing this pole-hiding issue.

Furthermore, the coefficients $\alpha_j$ may be found by imposing additional interpolation conditions, either exactly or in a least-squares sense:
\begin{equation*}
    \phi(\widetilde k_i)=\widetilde\phi_{\max,n_b}(\widetilde k_i)\quad\textup{for all }i=M_K'+1,\ldots,M_K'+n_o.
\end{equation*}
This procedure has the same complexity as the ``sum-based'' one.

The benefit of the expression $\widetilde\phi_{\max,0}$, compared to $\widetilde\phi_{\textup{sum},0}$, is that it is \emph{not} globally smooth, but only piecewise differentiable. As our numerical results in \cref{sec:numexp} show, this feature allows $\widetilde\phi_{\max,n_b}$ to better mimic the regularity of the exact field amplification $\phi$, which is also only piecewise smooth. We still have convergence.

\begin{corollary}
    \Cref{prop:conv} also holds with $\widetilde\phi_{\max,n_b}$ in place of $\widetilde\phi_{\textup{sum},n_b}$.
\end{corollary}
\begin{proof}
    The proof of \cref{prop:conv} still applies, since $\widetilde\phi_{\textup{max},0}$, cf.~\cref{eq:asmax}, is continuous over $K$.
\end{proof}

\subsection{The overall algorithm}\label{sec:sub:hybrid}
The overall hybrid surrogate-modeling-based procedure is summarized in \cref{algo:hybrid}. Exploiting the fact that $M_K'=\bigO{\nK}$, we can conclude that, under our working assumption that $\nK\ll \Nh$, the total complexity of the method is $\bigO{\Nh^3(\nK+n_o)+(\nK+n_b)\NK}$, with $c_{\textup{surrogate}}=\bigO{\nK+n_b}$. If $n_b\sim n_o=\bigO{\nK}$, which is our recommendation, these complexities simplify to those shown in the last row of \cref{tab:complexity}, which includes a comparison of the complexities of all discussed algorithms.

In particular, note that, asymptotically in $N_h$, building a hybrid approximation $\widetilde\phi\in\{\widetilde\phi_{\textup{sum},n_b},\widetilde\phi_{\max,n_b}\}$ is as costly as building a matrix-valued rational approximation via \cref{algo:rat}. However, the evaluation of a hybrid $\widetilde\phi$ is $N_h$-independent, in contrast to the $\bigO{N_h^2}$-cost of evaluating the surrogate built with \cref{algo:rat}.

This enhanced ``online efficiency'' (using surrogate modeling jargon) is confirmed in our numerical tests.

\begin{algorithm}[ht]
	\caption{Approximation of $\{\phi(k_i)\}_{i=1}^{\NK}$ by hybrid rational surrogate modeling}
	\begin{algorithmic}\label{algo:hybrid}
        \STATE{Compute $\uM^{1/2}$\hfill}\COMMENT{cost $\bigO{\Nh^3}$}
        \STATE{Draw the random vector $\vbh\in\bC^{2\Nh}$ with i.i.d.~Gaussian entries}
        \STATE{Select $\nK$ sample points $\{\overline k_i\}_{i=1}^{\nK}\subset\bC$, e.g., adaptively \cite{pradovera22,pradovera23,aaacontinuum}}
        \FOR{$\overline k\in\{\overline k_i\}_{i=1}^{\nK}$}
            \STATE{Compute sample $\uC{\overline k}\vbh=\uM^{1/2}\uA{\overline k}^{-1}\uM^{1/2}\vbh$\hfill}\COMMENT{cost $\bigO{\Nh^3}$}
        \ENDFOR
        \STATE{Get the weights $\{w_i\}_{i=1}^{\nK}$ by solving a data-fit problem\hfill}\COMMENT{cost $\bigO{\Nh \nK^3}$, cf.~\cref{rem:sketchedweights}}
        \STATE{Assemble the rational surrogate $k\mapsto\vell(k)\approx\uCk\vbh$ as in \cref{eq:barymatrix}}
        \STATE{Compute the rational surrogate's poles $\widetilde\Lambda_K$\hfill}\COMMENT{cost $\bigO{\nK^3}$, cf.~\cite[eqn.~(3.11)]{aaa}}
        \STATE{Select collocation points $\{\widetilde k_i\}_{i=1}^{M_K'+n_o}$}
        \FOR{$\widetilde k\in\{\widetilde k_i\}_{i=1}^{M_K'+n_o}$}
            \STATE{Evaluate $\phi(\widetilde k)$\hfill}\COMMENT{cost $\bigO{\Nh^3}$}
        \ENDFOR
        \STATE{Find the coefficients $\widetilde\psi_{\lambda}$ and $\alpha_j$ of either $\widetilde\phi_{\textup{sum},n_b}$ or $\widetilde\phi_{\max,n_b}$ as in \cref{sec:sub:fit}, cf.~\cref{eq:assum,eq:asmax}\\\hfill}\COMMENT{cost $\bigO{(\nK+n_o)(\nK+n_b)^2}$}
        \FOR{$k\in\{k_i\}_{i=1}^{\NK}$}
            \STATE{Evaluate either $\widetilde\phi(k)=\widetilde\phi_{\textup{sum},n_b}(k)$ or $\widetilde\phi(k)=\widetilde\phi_{\max,n_b}(k)$\hfill}\COMMENT{cost $\bigO{\nK+n_b}$}
        \ENDFOR
        \RETURN{$\{\widetilde\phi(k_i)\}_{i=1}^{\NK}$}
	\end{algorithmic}
\end{algorithm}

\begin{table}[bth]
    \centering
    \renewcommand{\arraystretch}{1.1}
    \begin{tabular}{c|c|c|c|c|}
        & Number of & \multicolumn{2}{c|}{Asymptotic complexity} & Main sources of\\
        \cline{3-4}
        & samples & Training & Evaluation (each) & inaccuracy\\
        \hline
        \cref{algo:direct} & $\NK$ & --- & $\Nh^3$ & ---\\
        \hline
        \cref{algo:rat} & $\nK\sim M_K$ & $\Nh^3\nK+\Nh^2\nK^3$ & $\Nh^2(\nK+\log\Nh)$ & rational approx.\\
        \hline
        \multirow{2}{*}{\cref{algo:sketch}} & $\hatnK\gtrsim M_Kq$ & \multirow{2}{*}{$\Nh^3\hatnK+\Nh\hatnK^3$} & \multirow{2}{*}{$\hatnK^2$} & rat.~approx. of\\
        & $q\sim\log\Nh$ & & & non-merom.~func.\\
        \hline
        \multirow{2}{*}{\cref{algo:hybrid}} & $\nK\sim M_K$ & \multirow{2}{*}{$\Nh^3\nK$} & \multirow{2}{*}{$\nK$} & rat.~approx.~\& \\
         & $n_o=\bigO{\nK}$ & & & error in \cref{eq:assumpoly} or \cref{eq:asmax}\\
        \hline
    \end{tabular}
    \caption{Summary of complexities of the presented algorithms when $M_K\ll \Nh$.}
    \label{tab:complexity}
\end{table}

\section{Numerical results}\label{sec:numexp}
In this section, we perform numerical tests to assess the effectiveness of our proposed methods on some 2D benchmarks. For the BE discretization, we use continuous piecewise-linear polynomials for both Dirichlet and Neumann traces. Our algorithms are coded in Python 3.10, with wrappers to interface the code with a C++ implementation of the P1-BEM, compiled with \texttt{c++} 11.4.0. For reproducibility, our code is made available at \url{https://github.com/pradovera/hybridScatteringAmplification.git}, where further information on software is also provided. All tests are performed on a machine with a 6-core 12-thread 3.5-GHz Intel\textsuperscript{\textregistered} Xeon\textsuperscript{\textregistered} processor. The BEM implementation takes full advantage of the multi-threaded architecture through OpenMP. Our tests involve the three scatterers displayed in \cref{fig:shapes}.

\begin{figure}[htb]
    \centering
    \includegraphics{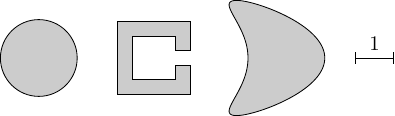}
    \caption{Disk, C-shaped, and kite scatterers.}
    \label{fig:shapes}
\end{figure}

\subsection{Disk scatterer}
We consider a simple unit-disk scatterer $\Omega=B_1(\mathbf{0})$ with contrast $\nin=\sqrt{20}$, cf.~\cref{eq:refraction}. We target the approximation of the field amplification over the wavenumber range $K=[1,3]$, over a uniform grid with $\NK=1001$ frequency points. The BEM discretization uses an 11\textsuperscript{th} order quadrature rule to approximate the integral operators on $\Nh=200$ ``panels''. (Other resolutions $\Nh\in\{50,100\}$ were also tested with similar results.) The number of RPM iterations used for estimating $\phi$ is $q=\lfloor10(1+\log \Nh)\rfloor=70$. For rational approximation of matrix- and vector-valued functions, we employ an $hp$-greedy version of minimal rational interpolation, as proposed in \cite{pradovera22}, with a tolerance of $1\%$ on the relative rational-approximation error in the Frobenius norm.

\subsubsection{Spectral features}\label{sec:numexp:disk:lambda}
\begin{figure}[tb]
    \centering
    \includegraphics[width=.95\textwidth]{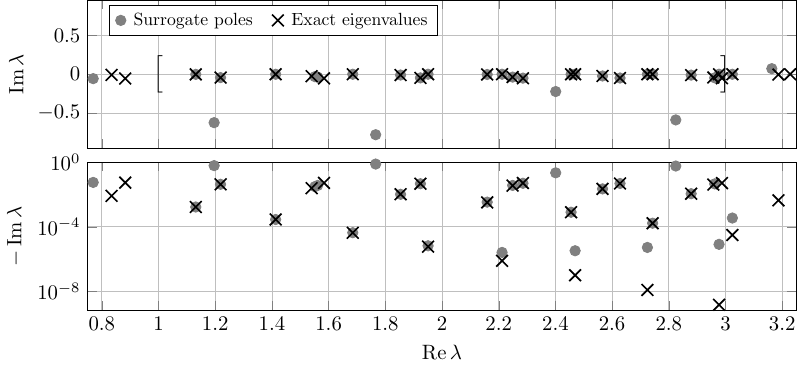}
    \caption{Eigenvalues of the problem with disk scatterer (black crosses). The gray circles are poles of the rational function $\uLz$ obtained with \cref{algo:rat}. In the bottom logarithmic-scale plot, poles closer to the real axis are further down.}
    \label{fig:circlepoles}
\end{figure}

In \cref{fig:circlepoles}, we show the complex eigenvalues of the problem, which, given the symmetry of the problem, are available analytically as described in \cref{sec:disk}. The exact eigenvalues are shown as crosses in \cref{fig:circlepoles}. From the bottom plot, which is logarithmic in the negative imaginary parts of the eigenvalues, we can observe sequences of eigenvalues converging to the real axis at exponential rate. By \cref{prop:asympt}, these correspond to spikes in the field amplification $\phi$.

In \cref{fig:circlepoles}, we also include the poles of the rational approximation obtained by \cref{algo:rat} as gray circles. We can see that the rational poles provide a fairly good approximation of some exact eigenvalues. The not-so-high accuracy is a straightforward consequence of the fact that the rational approximation is built based only on real-axis information. Accordingly, some errors are incurred in the estimation of the eigenvalues' locations, especially for those with larger (in magnitude) imaginary parts.

We can also observe how the rational approximation introduces an error in the eigenvalues that lie closest to the real axis, cf.~the eigenvalues with negative imaginary parts smaller than $10^{-4}$. This might seem surprising, since the real information should do a better-than-average job at identifying ``almost real'' eigenvalues. However, the (apparently) large errors are, in fact, rather small, magnified by the logarithmic scale. These are just consequences of the BEM discretization: refining the BEM mesh leads to a more accurate approximation of the almost-real eigenvalues.

Moreover, we can see that one of the rational poles, $\lambda\approx3.16+0.07\textup{i}$, has positive imaginary part. This is a non-physical eigenvalue, but should not be necessarily considered spurious, cf.~\cref{rem:spurious}. Indeed, rational approximation sometimes relies on unstable poles to improve approximation accuracy. It is commonly suggested to either keep such non-physical poles, to ``flip'' them \cite{gustavsen99}, i.e., change the sign of their imaginary parts, or to remove them \cite{huybrechs23}. Four more spurious poles can also be observed in \cref{fig:circlepoles} (top), along a band with imaginary part between $-1$ and $-0.2$.

\subsubsection{Field amplification}
\begin{figure}[tb]
    \centering
    \includegraphics[width=.95\textwidth]{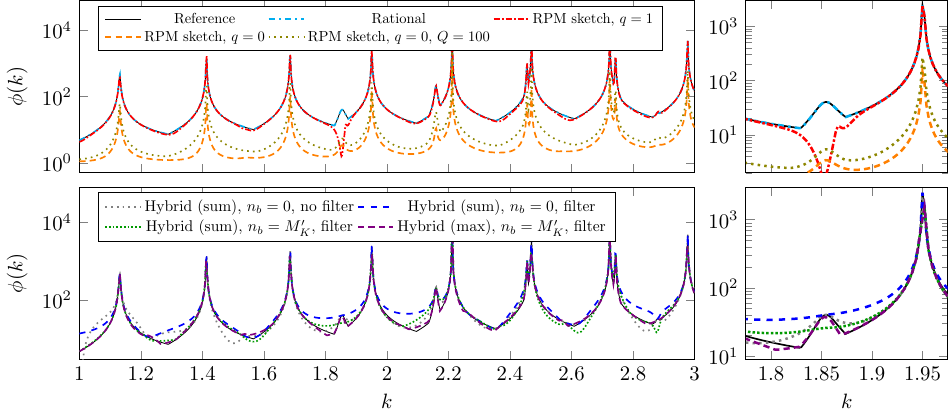}
    \caption{Field amplification for disk scatterer using \cref{algo:rat,algo:sketch} (top) and \cref{algo:hybrid} (bottom). The reference results obtained with \cref{algo:direct} are also included as a black solid line. The right plots are zoomed views.}
    \label{fig:circle_app}
\end{figure}

We display the results obtained with all our algorithms in \cref{fig:circle_app}. We will take the results of \cref{algo:direct} as reference. Many almost-resonances can be observed as quasi-asymptotes of $\phi$, confirming our conclusions based on \cref{fig:circlepoles}. We can observe a good approximation accuracy with \cref{algo:rat}, with the rational approximation requiring only $\nK=34$ adaptively selected samples. This is as expected, since the rational approximation \cref{eq:barymatrix} of the meromorphic function $\uCk$ is very efficient.

The RPM-based procedure is more unwieldy. We could not make the rational approximation in \cref{algo:sketch} converge for a ``large enough'' number of RPM iterations $q=70$, nor, in fact, for any $q\geq2$. This is a symptom of the non-meromorphy of the functions \cref{eq:barysketched}, which makes the number of needed sample points drastically increase with $q$. The number of samples is large even for $q=1$, the largest value for which we could make minimal rational approximation converge: the number of adaptively selected samples in that case is $\nK=55$ and $\nK=93$ for $p=0$ and $p=1$ in \cref{eq:barysketched}, respectively. The corresponding surrogate is fairly accurate, although the error is larger at a few locations, cf.~the top-right plot in \cref{fig:circle_app}.

Choosing $q=0$ in RPM, as suggested in \cref{rem:nonmero} to ease the approximation effort, results in a gross underestimation of $\phi$, since the RPM estimate \cref{eq:sketch} has not yet converged to the exact spectral norm. Still, the locations of the peaks of $\phi$ are identified well. One might think that the results could be improved by increasing the number of sketched directions, replacing the random vector $\vbh$ with a $((2\Nh)\times Q)$-matrix $\uBh$ with random i.i.d. entries. However, our experiments in this direction (see \cref{fig:circle_app} for the results with $Q=100$) resulted in very marginal improvements. These failures can be explained by the ``concentration of measure'' phenomenon, as encoded, e.g., by the Johnson-Lindenstrauss lemma \cite{JL}, which predicts a good approximation accuracy only if at least $Q\sim e^{2\Nh}$ random vectors are considered.

The results obtained with the hybrid approach in \cref{algo:hybrid} are, by design, less accurate than those of \cref{algo:rat}, although very good accuracy is retained near the peaks of $\phi$. The hybrid approach starts with the rational approximation of a ``sketched'' version of $\uCk$, which requires $\nK=35$ adaptively selected samples. This number of samples is close to that of \cref{algo:rat} for the approximation of $\uCk$ directly, since the rational-approximation problems solved by \cref{algo:rat,algo:hybrid} have roughly the same ``difficulty''.

Starting from the case \cref{eq:assum} without extra terms ($n_o=n_b=0$), we can observe that the results change depending on whether potentially spurious poles are filtered as in \cref{eq:polefilter}. Specifically, in this experiment, by filtering out $14$ poles in $\widetilde\Lambda_K^{\textup{all}}\setminus\widetilde\Lambda_K^{\textup{close}}$, we manage to recover an approximation that acts as a uniform upper-bound to the exact $\phi$, in accordance with \cref{eq:assum}, at the cost of some accuracy. This may be a very desirable property, depending on the application.

A partial explanation for this effect can be found by looking at the shape of the $14$ basis functions $\abs{\cdot-\lambda}^{-1}$, cf.~\cref{sec:sub:fit}, that are removed when poles are filtered. As a specific example, the basis function corresponding to the spurious pole $\lambda_{1}\approx1.20-0.62\textup{i}$, namely, $((k-1.20)^2+0.62^2)^{-1/2}$, is more ``spread out'' than the basis function corresponding to the non-spurious pole $\lambda_{2}\approx1.22-0.04\textup{i}$, whose real part is similar but whose imaginary part is much smaller, namely, $((k-1.22)^2+0.04^2)^{-1/2}$. Due to its larger variance, the former basis function has a global effect on the approximation $\widetilde\phi_{\textup{sum},0}$, breaking the ``upper-bound'' property of \cref{eq:assum} at wavenumbers far from $\real\lambda_{2}$.

We also test the hybrid algorithm with $n_b>0$ extra basis elements, by taking a piecewise-linear finite element basis on a uniform partition of $K$ into $(n_b-1)$ intervals. In our experiment, we set $n_o=n_b=M_K'=20$, the number of (filtered) rational poles. This leads to an enhanced approximation accuracy. Interestingly, the ``max'' approximation $\widetilde\phi_{\textup{max},n_b}$, defined in \cref{eq:asmax}, seems to perform better than the ``sum'' one $\widetilde\phi_{\textup{sum},n_b}$, defined in \cref{eq:assumpoly}.

\subsubsection{Approximation errors}
\begin{figure}[tb]
    \centering
    \includegraphics[width=.95\textwidth]{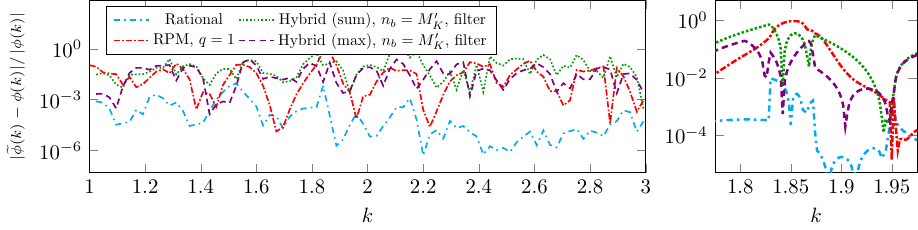}
    \caption{Relative error in field amplification for the disk. The reference results have been obtained with \cref{algo:direct}. The data in the left plot has been coarsened for readability. The right plot is a zoomed view.}
    \label{fig:circle_error}
\end{figure}

The approximation errors are shown in \cref{fig:circle_error}, again taking the results of \cref{algo:direct} as reference. From \cref{fig:circle_error}, we can draw quantitative conclusions confirming our observations above. In particular, we can see how the relative approximation error achieved by \cref{algo:rat} is uniformly below the prescribed $10^{-2}$ tolerance, although the tolerance is imposed on $\uCk$ (in the Frobenius norm) rather than on $\phi$.

To get a more concise error measure, we also compute the root-mean-square relative error
\begin{equation*}
    \left(\frac1{\NK}\sum_{i=1}^{\NK}\frac{|\widetilde\phi(k_i)-\phi(k_i)|^2}{\abs{\phi(k_i)}^2}\right)^{1/2}.
\end{equation*}
This is reported in the last column of \cref{tab:disk}. Therein, we can see further confirmation of our conclusions. Notably, $\widetilde\phi_{\textup{max},M_K'}$ achieves an approximation error more than 3 times lower than $\widetilde\phi_{\textup{sum},M_K'}$ does.

\begin{table}[tb]
    \centering
    \renewcommand{\arraystretch}{1.1}
    \begin{tabular}{c|c|c|c||c|c|c||c|}
        \multicolumn{3}{c|}{} & \multirow{2}{*}{$\nK$} & \multicolumn{2}{c|}{Surrogate training [s]} & Evaluation & RMS rel.\\
        \cline{5-6}
        \multicolumn{3}{c|}{} & & rat.~app. & extra & [s] & error\\
        \hline
        \multicolumn{3}{c|}{\cref{algo:direct}} & --- & --- & --- & $2.26\cdot10^4$ & ---\\
        \hline
        \multicolumn{3}{c|}{\cref{algo:rat}} & $34$ & $7.87\cdot10^2$ & --- & $9.55\cdot10^0$ & $0.2\%$\\
        \hline
        \multicolumn{2}{c|}{\multirow{2}{*}{Alg.~\ref{algo:sketch}}} & $q=1$ & $55+93$ & $3.42\cdot10^3$ & --- & $1.54\cdot10^{-1}$ & $12.6\%$\\
        \cline{3-8}
        \multicolumn{2}{c|}{} & $q=0$ & $35$ & $8.15\cdot10^2$ & --- & $1.29\cdot10^{-1}$ & $91.4\%$\\
        \hline
        \multirow{4}{*}{Alg.~\ref{algo:hybrid}} & \multirow{2}{*}{sum} & $n_o=0$ & \multirow{4}{*}{$35$} & \multirow{4}{*}{$8.11\cdot10^2$} & $4.66\cdot10^2$ & $3.74\cdot10^{-3}$ & $63.6\%$\\
        \cline{3-3}\cline{6-8}
        & & \multirow{2}{*}{$n_o=M_K'=20$} & & & \multirow{2}{*}{$9.40\cdot10^2$} & $6.21\cdot10^{-3}$ & $25.0\%$\\
        \cline{2-2}\cline{7-8}
        & \multirow{2}{*}{max} & & & & & $5.94\cdot10^{-3}$ & $7.4\%$\\
        \cline{3-3}\cline{6-8}
        & & $n_o=2M_K'=40$ & & & $1.38\cdot10^3$ & $8.09\cdot10^{-3}$ & $4.5\%$\\
        \hline
    \end{tabular}
    \caption{Timings for disk scatterer. The ``extra'' step relates to taking the $(M_K'+n_o)$ samples needed to set up problem \cref{eq:assumfitpoly}. The last column shows the root-mean-square (RMS) relative approximation errors.}
    \label{tab:disk}
\end{table}

For completeness, we also test $\widetilde\phi_{\textup{max},2M_K'}$, which, compared to $\widetilde\phi_{\textup{max},M_K'}$, is characterized by a $33\%$ increase in the number of terms in \cref{eq:asmax} and, consequently, in the cost of building and evaluating the surrogate. The corresponding approximation error decreases by a significant amount, while retaining an extremely low surrogate-evaluation cost.

\subsubsection{Timings}
The timings of the different approaches are summarized in \cref{tab:disk}, where they are split into training and evaluation phases. The total runtime is the sum of the three reported times. In this experiment, \cref{algo:rat} seems to be the most effective, achieving the lowest error out of all methods employing surrogates, with also the lowest total computation time. Note, in particular, how the evaluation of the surrogate field amplification is about 3 orders of magnitude faster than its direct computation by the BEM, although both methods must find the largest singular value of a $(2\Nh)\times(2\Nh)$ matrix. This is because the most expensive steps in the BEM are matrix assembly and inversion, which are avoided through the use of the rational surrogate in the evaluation step.

Moreover, although the training phase of the hybrid approach is about twice as costly, it leads to an incredibly cheap surrogate, as evidenced by the extremely low evaluation time. Overall, \cref{algo:hybrid} might be superior to \cref{algo:rat} when $\Nh$ or $\NK$ are very large, compared to $\nK$, as long as one is willing to accept a slightly higher approximation error.

\begin{remark}\label{rem:compare}
    As a further comparison, we also ran the algorithm introduced in \cite{grubisic23}, whose objective is to identify the locations of the peaks of $\phi$. This is achieved by constructing a piecewise-polynomial surrogate model for $\phi(\cdot)^{-1}$, relying on an $h$-adaptive procedure (with respect to frequency) that targets all local maxima of $\phi$. We ran the algorithm over the wavenumber range $K$, using the parameters suggested in \cite{grubisic23}.

    In total, the algorithm required $\nK=196$ samples to converge, a much higher number than for our rational-approximation routines, since the latter better exploit the meromorphic structure of the problem. In addition, note that \cite{grubisic23} requires expensively computing derivatives of $\phi(\cdot)^{-1}$ at each sample point. As such, the cost of each sample, although still cubic in complexity, is higher than just sampling $\phi$ by more than an order of magnitude. The increased per-sample cost made the training phase of the algorithm last longer than $2\cdot10^5$ seconds, an order of magnitude more than our baseline, \cref{algo:direct}.

    Moreover, the method from \cite{grubisic23} was worse than our best approaches, namely, \cref{algo:rat,algo:hybrid}, also in terms of approximation accuracy, with particularly large errors far away from the peaks of $\phi$. This is a consequence of the fact that, by design, the method from \cite{grubisic23} targets only the peaks of the field amplification, neglecting regions where the behavior of $\phi$ is more ``tame''.
\end{remark}

\subsection{C-shaped and kite-shaped scatterer}
Now we move to two tests involving more complex scatterers, with the ``C'' and ``kite'' shapes shown in \cref{fig:shapes} (center and right). We now target a larger wavenumber range $K=[1,5]$, over a uniform grid with $\NK=2001$ points. All other parameters are as in the previous experiment. Taking \cref{algo:direct} as reference, we only test the most promising approaches, namely, \cref{algo:rat,algo:hybrid}. In particular, we test the latter hybrid approach with (i) filtered poles, (ii) $n_o=n_b=M_K'$, (iii) both ``sum'' and ``max'' flavors.

\begin{figure}[tb]
    \centering
    \includegraphics[width=.95\textwidth]{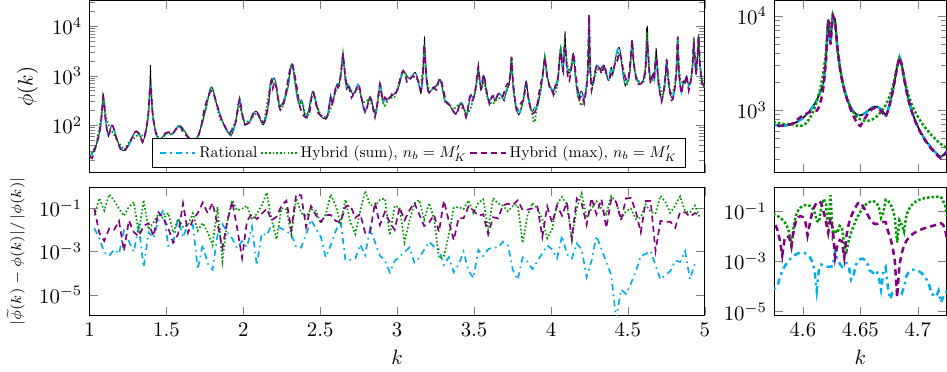}
    \caption{Top row: field amplification for C-shaped scatterer using \cref{algo:rat} and \cref{algo:hybrid} with $n_b=M_K'$. The reference results obtained with \cref{algo:direct} are also included as a black solid line. Bottom row: relative approximation error. The data in the bottom left plot has been coarsened for readability. The right plots are zoomed views.}
    \label{fig:cshape_app}
\end{figure}

\begin{figure}[tb]
    \centering
    \includegraphics[width=.95\textwidth]{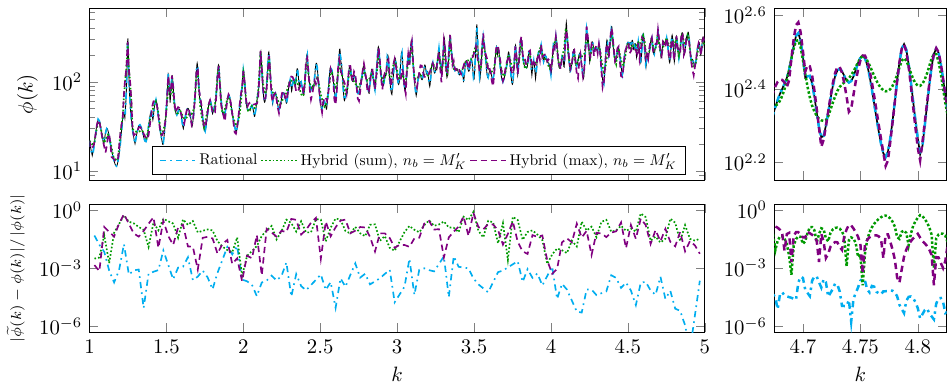}
    \caption{Top row: field amplification for kite scatterer using \cref{algo:rat} and \cref{algo:hybrid} with $n_b=M_K'$. The reference results obtained with \cref{algo:direct} are also included as a black solid line. Bottom row: relative approximation error. The data in the bottom left plot has been coarsened for readability. The right plots are zoomed views.}
    \label{fig:kite_app}
\end{figure}

The results for the ``C'' and ``kite'' shapes are displayed in \cref{fig:cshape_app,fig:kite_app}, respectively. Therein, we can observe how both scatterers yield much more complicated field amplification profiles than the disk's. The behavior of $\phi$ is especially wild for the kite, notably, by a scaling argument, because of its larger area. This complexity also emerges in the form of the large(r) numbers $\nK$ of samples that the adaptive rational-approximation routine must take to achieve the prescribed $10^{-2}$ relative tolerance: $\nK=127$ for the ``C'' shape and $\nK=196$ for the kite.

The accuracies of our methods are consistent with those observed in our first experiment: the ``max'' version of the hybrid algorithm slightly improves on the ``sum'' version, although the improvement is not as striking as with the disk scatterer. The approach based on matrix rational approximation is more accurate than both hybrid methods by at least an order of magnitude. The root-mean-square approximation errors, reported in the last column of \cref{tab:shapes}, confirm these observations.

\begin{table}[tb]
    \centering
    \renewcommand{\arraystretch}{1.1}
    \begin{tabular}{c|c|c|c|c||c|c|c||c|}
        \multicolumn{4}{c|}{} & \multirow{2}{*}{$\nK$} & \multicolumn{2}{c|}{Surrogate training [s]} & Evaluation & RMS rel.\\
        \cline{6-7}
        \multicolumn{4}{c|}{} & & rat.~app. & extra & (total) [s] & error\\
        \hline
        \multirow{5}{*}{\rotatebox{90}{C shape\hspace{.2em}}} & \multicolumn{3}{c|}{\cref{algo:direct}} & --- & --- & --- & $3.60\cdot10^4$ & ---\\
        \cline{2-9}
        & \multicolumn{3}{c|}{\cref{algo:rat}} & $127$ & $2.39\cdot10^3$ & --- & $3.29\cdot10^1$ & $1.0\%$\\
        \cline{2-9}
        & \multirow{3}{*}{Alg.~\ref{algo:hybrid}} & sum & \multirow{2}{*}{$n_o=M_K'=89$} & \multirow{3}{*}{$150$} & \multirow{3}{*}{$2.80\cdot10^3$} & \multirow{2}{*}{$3.31\cdot10^3$} & $8.52\cdot10^{-3}$ & $16.2\%$\\
        \cline{3-3}\cline{8-9}
        & & \multirow{2}{*}{max} & & & & & $8.25\cdot10^{-3}$ & $12.0\%$\\
        \cline{4-4}\cline{7-9}
        & & & $n_o=2M_K'=178$ & & & $4.98\cdot10^3$ & $1.30\cdot10^{-2}$ & $5.6\%$\\
        \hline
        \multirow{5}{*}{\rotatebox{90}{kite\hspace{.2em}}} & \multicolumn{3}{c|}{\cref{algo:direct}} & --- & --- & --- & $3.21\cdot10^4$ & ---\\
        \cline{2-9}
        & \multicolumn{3}{c|}{\cref{algo:rat}} & $196$ & $3.28\cdot10^3$ & --- & $3.56\cdot10^1$ & $0.5\%$\\
        \cline{2-9}
        & \multirow{3}{*}{Alg.~\ref{algo:hybrid}} & sum & \multirow{2}{*}{$n_o=M_K'=130$} & \multirow{3}{*}{$194$} & \multirow{3}{*}{$3.17\cdot10^3$} & \multirow{2}{*}{$4.30\cdot10^3$} & $8.89\cdot10^{-3}$ & $18.7\%$\\
        \cline{3-3}\cline{8-9}
        & & \multirow{2}{*}{max} & & & & & $8.24\cdot10^{-3}$ & $16.4\%$\\
        \cline{4-4}\cline{7-9}
        & & & $n_o=2M_K'=260$ & & & $6.47\cdot10^3$ & $1.37\cdot10^{-2}$ & $7.7\%$\\
        \hline
    \end{tabular}
    \caption{Timings for C-shaped and kite scatterers. The ``extra'' step relates to taking the $(M_K'+n_o)$ samples needed to set up problem \cref{eq:assumfitpoly}. The last column shows the root-mean-square (RMS) relative approximation errors.}
    \label{tab:shapes}
\end{table}

The methods' timings, reported in \cref{tab:shapes}, are also consistent with those measured for the disk scatterer. Once more, looking only at the evaluation of $\phi$ and $\widetilde\phi$, we see that \cref{algo:rat,algo:hybrid} are about $10^3$ and $10^6$ times faster, respectively, than the direct evaluation in \cref{algo:direct}. The speed-up obtained with the hybrid method is quite impressive, especially considering the accuracy attained. In all cases, building the surrogate model is about an order of magnitude faster than evaluating the field amplification directly.

\section{Conclusions}
In this work, we have introduced several methods for approximating the field amplification $\phi$ related to scattering problems. Through rational approximation, we are able to exploit the problem's meromorphic structure to great effect, leading to considerable computational savings. Among the presented algorithms, the ones that show the greatest potential are \cref{algo:rat,algo:hybrid}.
\begin{itemize}
\item \cref{algo:rat} performs rational approximation of the BEM-discretized solution operator, achieving great accuracy at the cost of a super-quadratic surrogate-evaluation cost due to the need to evaluate the spectral norm of large matrices.
\item The ``hybrid'' method from \cref{algo:hybrid} targets the field amplification more directly, through a scalar interpolatory approach. This reduces the evaluation cost at the price of some accuracy, while keeping the training cost comparable to that of \cref{algo:rat}. The hybrid method uses rational approximation as a tool to greatly improve the approximation quality and efficiency, by enriching the interpolation basis with some additional rational terms. This strategy ultimately enables a modest approximation accuracy with an evaluation cost that is much lower than that of \cref{algo:rat}.
\end{itemize}

Both rational and hybrid approaches yield excellent results in our three test cases, especially when looking at the identification of the locations of the almost-resonances, namely, the spikes in the field amplification. In our view, it is important to note how well all our approaches perform, despite their very limited ``intrusiveness'': none of the four methods presented requires access to the structure of the BEM operators, namely, to their specific frequency dependence. This feature is what ultimately enables applying our methods to a problem whose parameter dependence is so intrinsically nonlinear. Moreover, we highlight how our methods are able to succeed at approximating field amplification despite being limited to sampling real wavenumbers only.

Concerning future research directions, we note that the field amplification has been defined in terms of \emph{surface potentials} on the interface $\Gamma=\partial\Omega$. In some cases, it might be of greater interest to consider \emph{volume} field amplifications, e.g., measured in the $\Honein$-metric: $\phi_\Omega(k):=\sup_{\uinc\neq0}\norm{\utot(k)}_{\Honein}/\norm{\uinc}_{\Honein}$, cf.~\cref{eq:strongform}. In our setting, $\phi_\Omega$ may be computed via the representation formula \cref{eq:green}, which provides an \emph{invertible} extension operator from surface potentials to Helmholtz solutions over $\bR^d$. In principle, our discussion and our techniques should apply also to $\phi_\Omega$. In particular, our methods should be able to well approximate the locations of the problem's quasi-resonances, which appear as ``peaks'' both in $\phi$ and $\phi_\Omega$. However, in such case the approximation task is much more challenging, since (i) the wavenumber-dependent extension operator makes $\phi_\Omega$ depend on $k$ in a more complicated way than $\phi$, and (ii) building and inverting the extension operator is difficult in itself. For these reason, further analysis and implementation work remain to be done.

\appendix

\section{Exact scattering eigenvalues on the disk}\label{sec:disk}
Let $J_\nu$ and $H_\nu$ denote Bessel and Hankel functions of the first kind of index $\nu\in\mathbb Z$, respectively. Using polar coordinates, the eigenfunctions of the scattering problem for the unit disk $B_1(\mathbf{0})$ with refraction index $n_i$ are of the form
\begin{equation*}
    u_\nu(r,\theta)=\begin{cases}
        \alpha J_\nu(\nin\lambda r)e^{\textup{i}\nu\theta}\quad&\textup{if }r<1,\\
        \beta H_\nu(\lambda r)e^{\textup{i}\nu\theta}\quad&\textup{if }r>1,
    \end{cases}
\end{equation*}
using the integer $\nu\in\mathbb Z$ to index the spectrum. The above scalars $\lambda$, $\alpha$, and $\beta$ solve the following nonlinear eigenvalue problem:
\begin{equation*}
    \textup{find }\lambda\in\bC\textup{ and }(\alpha,\beta)\in\bC^2\setminus\{\mathbf0\}\textup{ s.t.}\begin{pmatrix}
        J_\nu(\nin\lambda) & H_\nu(\lambda)\\
        \nin\lambda J_\nu'(\nin\lambda) & \lambda H_\nu'(\lambda)
    \end{pmatrix}\begin{bmatrix}
        \alpha\\ -\beta
    \end{bmatrix}=\mathbf0.
\end{equation*}
The above $2\times 2$ transcendental system encodes compatibility conditions of Dirichlet and Neumann traces across the interface $\Gamma$. In our numerical tests, we solve this nonlinear eigenvalue problem by using Beyn's method \cite{beyn12}.

\section*{Data availability}

The C++ and Python source code to reproduce this paper's numerical experiments can be downloaded from \url{https://github.com/pradovera/hybridScatteringAmplification.git}.

\bibliographystyle{abbrvurl}
\bibliography{bibliography}

\end{document}